\documentclass[11pt]{amsart}

\usepackage{amssymb,mathrsfs,enumitem,amsmath,amsthm,bbold}
\usepackage[mathscr]{euscript}
\usepackage{color}

\usepackage[oldstyle]{libertine}

\usepackage[all]{xy}
\definecolor{Chocolat}{rgb}{0.36, 0.2, 0.09}
\definecolor{BleuTresFonce}{rgb}{0.215, 0.215, 0.36}
\definecolor{EgyptianBlue}{rgb}{0.06, 0.2, 0.65}

\usepackage[OT2,OT1]{fontenc}

\usepackage[backend=bibtex,
    sorting=anyt,
    isbn=false,
    url=false,
    doi=false,
    maxbibnames=100
    ]{biblatex}
\usepackage[colorlinks,final,hyperindex]{hyperref}
\hypersetup{citecolor=BleuTresFonce, urlcolor=EgyptianBlue, linkcolor=Chocolat}

\newtheorem*{itheorem}{Theorem}
\newtheorem{theorem}{Theorem}[section]
\newtheorem{corollary}[theorem]{Corollary}

\newtheorem{lemma}[theorem]{Lemma}
\newtheorem{proposition}[theorem]{Proposition}

\theoremstyle{definition}
\newtheorem{example}[theorem]{Example}
\newtheorem{remark}[theorem]{Remark}

\newcommand{\ac}{\scriptstyle \text{\rm !`}}

\DeclareMathAlphabet{\pazocal}{OMS}{zplm}{m}{n}

\def\calA{\pazocal{A}}

\def\calC{\pazocal{C}}

\def\calF{\pazocal{F}}

\def\calO{\pazocal{O}}
\def\calP{\pazocal{P}}

\def\calS{\pazocal{S}}
\def\calT{\pazocal{T}}
\def\calU{\pazocal{U}}

\def\calW{\pazocal{W}}

\def\o{{\mathsf{o}}}
\def\w{{\mathsf{w}}}

\DeclareMathOperator{\Lie}{Lie}
\DeclareMathOperator{\ULie}{\calU{}Lie}

\DeclareMathOperator{\Com}{Com}

\DeclareMathOperator{\Ass}{Ass}

\DeclareMathOperator{\uCom}{uCom}

\DeclareMathOperator{\uAss}{uAss}

\DeclareMathOperator{\PL}{PreLie}

\DeclareMathOperator{\id}{id}

\DeclareMathOperator{\tr}{tr}

\DeclareMathOperator{\Div}{div}

\DeclareMathOperator{\Cyc}{Cyc}

\DeclareMathOperator{\Fin}{\mathsf{Fin}}

\DeclareMathOperator{\Hom}{Hom}
\DeclareMathOperator{\End}{End}
\DeclareMathOperator{\WEnd}{WEnd}

\DeclareMathOperator{\Der}{Der}
\DeclareMathOperator{\SDer}{SDer}
\DeclareMathOperator{\Aut}{Aut}

\def\d{\mathrm{d}}

\DeclareMathAlphabet{\mathbbold}{U}{bbold}{m}{n}

\def\k{\mathbbold{k}}

\def\sfC{{\ensuremath{\mathsf{C}}}}

\begin{document}

\title[Stable homology and wheeled operads]{Stable homology of Lie algebras of derivations\\ and homotopy invariants of wheeled operads}

\author{Vladimir Dotsenko}

\address{ 
Institut de Recherche Math\'ematique Avanc\'ee, UMR 7501, Universit\'e de Strasbourg et CNRS, 7 rue Ren\'e-Descartes, 67000 Strasbourg CEDEX, France}

\email{vdotsenko@unistra.fr}

\dedicatory{To Boris Feigin, with gratitude and admiration}
\date{}

\begin{abstract}
We prove a theorem that computes, for any augmented operad $\calO$, the stable homology of the Lie algebra of derivations of the free algebra $\calO(V)$ with twisted bivariant coefficients (here stabilization occurs as $\dim(V)\to\infty$) out of the homology of the wheeled bar construction of~$\calO$; this can further be used to prove uniform mixed representation stability for the homology of the positive part of that Lie algebra with constant coefficients. This result generalizes both the Loday--Quillen--Tsygan theorem on the homology of the Lie algebra of infinite matrices and the Fuchs stability theorem for the homology of the Lie algebra of vector fields. We also prove analogous theorems for the Lie algebras of derivations with constant and zero divergence, in which case one has to consider the wheeled bar construction of the wheeled completion of~$\calO$. Similarly to how cyclic homology of an algebra $A$ may be viewed as an additive version of the algebraic $K$-theory of $A$, our results hint at the additive $K$-theoretic nature of the wheeled bar construction. 
\end{abstract}

\maketitle


\section{Introduction}

\subsection{Two classical results on Lie algebra homology}
The goal of this paper is to prove an ultimate common generalization of two results on the homology of infinite-dimensional Lie algebras. The first one is the well known result proved by Loday and Quillen \cite{MR0780077} and, independently, by Tsygan \cite{MR0695483} that concerns the homology of the Lie algebra $\mathfrak{gl}(A)=\varinjlim\mathfrak{gl}_n(A)$ for a unital associative algebra over 
$\mathbb{Q}$: it asserts that this homology is isomorphic to the free (super)commutative algebra on the shifted cyclic homology $HC_{\bullet-1}(A)$; there is also a similar result proved by Feigin and Tsygan \cite{MR0923136} and, independently, by Hanlon~\cite{MR0946439} for non-unital algebras. 
The second result is the Fuchs stability theorem~\cite{MR0725416} that concerns the homology of the Lie algebra $W_n$ of all derivations of $\mathbb{R}[x_1,\ldots,x_n]$ with coefficients in modules of tensor fields (of various degrees of covariance and contravariance), and the homology of its Lie subalgebra $L_1(n)$ consisting of derivations that send all generators to polynomials  vanishing to order two at the origin. In particular, that theorem asserts that, for fixed homological degree $k$ and internal degree $d$ and for all sufficiently large $n$, the $d$-homogeneous part of $H_k(L_1(n),\mathbb{R})$ vanishes for $k\ne d$. In modern language, the Lie algebra $L_1(n)$ is stably Koszul. The choice of the ground field in these two results being of no importance, we shall work over an arbitrary ground field $\k$ of zero characteristic.\\

\subsection{Derivations of free algebras}
While in classical terms $\mathfrak{gl}_n(A)$ is the Lie algebra of endomorphisms of the free $A$-module of rank $n$, and the Lie algebra $L_1(n)$ consists of certain derivations of the polynomial algebra in $n$ variables, a reader who is sufficiently used to operads can make the connection between them much more precise. If we view the algebra $A$ as an operad with only unary operations, the free $n$-generated algebra over that operad is the free $A$-module of rank $n$, and endomorphisms of that module are precisely derivations of that free algebra; also, the polynomial algebra is a free algebra for the operad $\uCom$ of unital commutative associative algebras, and the Lie algebra $L_1(n)$ consists of certain derivations of that algebra. Thus, the right common context is that of derivations of free algebras with a sufficiently large number of generators. In that direction, it is tempting to suggest that a common generalization of the two theorems mentioned above is given by the celebrated work of Kontsevich \cite{MR1247289,MR1341841}, who even characterizes the spirit of his argument as ``somewhere in between Gelfand--Fuchs computations and cyclic homology''. However, work of Kontsevich studies derivations that preserve some version of symplectic structure; operadically, this suggests that the underlying algebraic structure is controlled by a cyclic operad. Most of the subsequent work in this direction did indeed focus on the cyclic operad situation; as a very non-exhaustive list of examples we refer the reader to Conant and Vogtmann~\cite{MR2026331}, Ginzburg \cite{MR1839485}, Hamilton and Lazarev \cite{MR2395368}, and Mahajan \cite{Mah02,Mah03}. 

In the final stages of preparation of the paper, we became aware of the beautiful paper \cite{MR1671737} by  Kapranov, where a non-cyclic version of the result of Kontsevich is briefly mentioned; since studying derivations of free operadic algebras was not the main goal of that work, the result is stated without a proof, and is not correct as stated. However, an important contribution of that remark of Kapranov is the definition of an object that one would now call the wheeled completion of an operad, several years before wheeled operads and PROPs were formally introduced in~\cite{MR2483835,MR2641194}, and a hint that wheeled operads are related to stable homology of Lie algebras of derivations; for the computation of Fuchs mentioned above, this was also observed in \cite{MR2608525}. Making this philosophy into a precise mathematical statement is the main goal of this paper. Along the way, we propose a new equivalent definition of a wheeled operad which appears much easier to digest. For the reader who is not sufficiently familiar with the notion of a wheeled operad, it would be useful to recall at this point that this notion generalizes that of an operad; it keeps track of traces of actions of elements of operadic algebras in arbitrary finite-dimensional modules.

\subsection{Homology of Lie algebras of derivations}
Our first result concerns the homology of the Lie algebra $\Der(\calO(V))$ of all derivations of the free algebra $\calO(V)$ over an augmented operad $\calO$. This Lie algebra contains a subalgebra $\mathfrak{gl}(V)$ of derivations sending every generator to a linear combination of generators, and the augmentation of $\calO$ induces a Lie algebra augmentation $\Der(\calO(V))\twoheadrightarrow \mathfrak{gl}(V)$. The kernel of that augmentation is denoted $\Der^+(\calO(V))$. One can use the augmentation to make every $\mathfrak{gl}(V)$-module into a $\Der(\calO(V))$-module, where the action is via the augmentation. In the following statement, we view $\calO$ as a trivial wheeled operad (with all traces equal to zero).
\begin{itheorem}[{Th.~\ref{th:main1}}]
Let $\calO$ be an augmented operad, and $V$ a finite-dimensional vector space. The vector spaces 
 \[
H_\bullet(\Der^+(\calO(V)),\Hom(V^{\otimes \bullet_1},V^{\otimes \bullet_2}))^{\mathfrak{gl}(V)}
 \]
stabilize as $\dim(V)\to\infty$, with the stable value given by the coPROP completion of the wheeled cooperad $H_\bullet(\mathsf{B}^{\circlearrowleft}(\calO))$. This isomorphism is natural with respect to operad morphisms.
\end{itheorem}
One may use Theorem~\ref{th:main1} to prove some vanishing results for the homology $H_\bullet(\Der^+(\calO(V)),\k)$, and, in principle, to 
determine explicitly the stable $\mathfrak{gl}(V)$-module decomposition of that homology. This shows that our result is indeed a common generalization of the Loday--Quillen--Tsygan theorem (as well as of its non-unital version of Feigin--Tsygan and Hanlon) and of the Fuchs stability theorem. 

In view of the previous result, knowing the wheeled bar homology of an operad is invaluable for computing the stable homology of Lie algebras of derivations of free algebras. This motivates our following result, allowing one to fully compute, in a certain class of cases, the wheeled bar homology of $\calO$ in terms of the derivative species $\partial(\calO)$. That latter species has two commuting structures, that of a right $\calO$-module, and that of a left $\Ass$-module. Let $\partial(\overline{\calO})_0$ be the (reduced) space of indecomposables for the right $\calO$-module structure, which naturally acquires a structure of a left $\Ass$-module, that is a twisted associative algebra (and hence its cyclic homology is defined).

\begin{itheorem}[{Th.~\ref{th:calchom}}] Let $\calO$ be an augmented operad for which $\partial(\calO)$ is free as a right $\calO$-module. We have 
 \[
H_\bullet(\mathsf{B}^{\circlearrowleft}(\calO))\cong H_\bullet(\mathsf{B}(\calO))\oplus HC_{\bullet-1}(\partial(\overline{\calO})_0).
 \]
\end{itheorem}

Note that freeness of $\partial(\calO)$ as a right $\calO$-module is precisely what controls the PBW property of multiplicative universal enveloping algebras of $\calO$-algebras \cite{MR4381941}.

\subsection{Homology of Lie algebras of divergence-free derivations}
If one uses the terminology of wheeled operads, the paper of Kapranov \cite{MR1671737} mentioned above asserts that the wheeled completion of the operadic bar construction plays a certain role in computing stable homology of Lie algebras of derivations of free algebras, but then the wrong claim about the homology of that wheeled completion is made. The problem is that taking the wheeled completion does not commute with computing homology and, in particular, Kapranov's conclusion that, for a Koszul operad $\calO$, the answer can be expressed in terms of the wheeled completion of the Koszul dual cooperad $\calO^{\ac}$, is only true for operads sufficiently similar to $\Com$, and is not true for most ``interesting'' operads (for instance, it is definitely not true for the Lie operad). In fact, if we consider an associative algebra $A$ as an operad concentrated in arity one, the theorem of Loday--Quillen--Tsygan shows that the commutator quotient $|A|:=A/[A,A]$ appears in the result, and if one hopes for some statement with a flavour of Koszul duality, one should at least find a way to remove that ``homotopically badly behaved'' piece. In the framework of Loday--Quillen--Tsygan, this corresponds to replacing $\mathfrak{gl}(A)$ with $\mathfrak{sl}(A)$, which is accomplished by considering matrices whose trace vanishes in the commutator quotient $A/[A,A]$. As we shall see, this gives one a very useful hint on how to proceed. \\ 

We have already seen that wheeled operads are relevant in this context, and, in the theory of wheeled operads, the wheeled completion $\calO^{\circlearrowleft}$  of an operad $\calO$ is relevant. What does the bar construction of that wheeled operad tell us about derivations of free $\calO$-algebras? It turns out to be possible to find a meaningful answer to this question involving the notion of divergence of a derivation of the free algebra. For polynomial algebras, derivations are polynomial vector fields, so the notion of divergence one naturally encounters is recognizable from differential geometry, while for $\mathfrak{gl}_n(A)$ it is natural to calculate the divergence as the image of the trace in the commutator quotient. A general notion that is suggested by these two particular cases  was recently defined by Powell~\cite{Powell21} for free algebras over arbitrary operads; previously, it was known some years in the cases of free Lie algebras (see, for example, Enomoto and Satoh~\cite{MR2846914} and Satoh~\cite{MR2864772}) and in the case of free associative algebras (see, for example, Alekseev, Kawazumi, Kuno and Naef \cite{MR3758425}). Below, we shall give an alternative presentation of that definition which we find much easier to digest and to use in our context.\\ 

Using the notion of divergence, we define more infinite-dimensional Lie algebras that will be relevant for our second main result. Specifically, we shall consider the Lie algebra $\SDer(\calO(V))$ of all derivations of the free algebra  with zero divergence; since we assume our operad  augmented, there is a Lie algebra map $\SDer(\calO(V))\to \mathfrak{gl}(V)$, whose kernel is denoted $\SDer^+(\calO(V))$. Various relationships between all the Lie algebras that have been defined so far are summarized by the diagram
 \[
 \xymatrix@M=6pt{
\SDer^+(\calO(V))\ar@{^{(}->}^{}[r]\ar@{^{(}->}^{}[d]& \Der^+(\calO(V))\ar@{^{(}->}^{}[d]&\\
\SDer(\calO(V))\ar@{^{(}->}^{}[r] &\Der(\calO(V))\ar@{->>}^{}[r]&\mathfrak{gl}(V)
 } 
 \]

Our second main result is as follows.

\begin{itheorem}[{Th.~\ref{th:main2}}]
Let $\calO$ be an augmented operad, and $V$ a finite-dimensional vector space. The vector spaces
 \[
H_\bullet(\SDer^+(\calO(V)),\Hom(V^{\otimes \bullet_1},V^{\otimes \bullet_2}))^{\mathfrak{gl}(V)}
 \] 
stabilize as $\dim(V)\to\infty$, with the stable value given by the coPROP completion of the wheeled cooperad $H_\bullet(\mathsf{B}^{\circlearrowleft}(\calO^{\circlearrowleft}))$. This isomorphism is natural with respect to operad morphisms.
\end{itheorem}

As above, one can use this result to prove some vanishing results for the homology $H_\bullet(\SDer^+(\calO(V)),\k)$, and, in principle, to 
determine explicitly the stable $\mathfrak{gl}(V)$-module decomposition of that homology. The advantage of this result is in the fact that, unlike what we saw in the case of all derivations, the wheeled completion $\calO^{\circlearrowleft}$ of a Koszul operad $\calO$ does often (though not always) happen to be Koszul as a wheeled operad. Our theorem implies that, if the wheeled operad $\calO^{\circlearrowleft}$ is Koszul, then the Lie algebra $\SDer^+(\calO(V))$ is Koszul in all weights less than $\frac13\dim(V)$; this is reminiscent of a similar result for Torelli Lie algebras proved by Felder, Naef and Willwacher \cite{MR4549961} and, independently, by Kupers and Randal-Williams \cite{MR4575357}. Note that it follows from the results of Powell~\cite{Powell21} that the stable homology of $\SDer^+(\calO(V))$ in degree one is concentrated in weight one for any binary operad $\calO$; our work allows one to prove a generalization of this statement concerning Koszulness in bounded weights. Let us also remark that the idea of~\cite{Powell21} to use functors on the category $\mathcal{S}(R)$ of split monomorphisms between finite rank free $R$-modules may be useful in our context: for instance, the notion of $k$-torsion in that category captures the number of extra variables one needs to adjoin for various isomorphisms to hold. This will be explored elsewhere.

\subsection*{Structure of the paper} This paper is organized as follows. In Section \ref{sec:recoll}, we recall the necessary definitions and terminology concerning twisted associative algebras, operads, universal multiplicative envelopes, and K\"ahler differentials. In Section \ref{sec:div-wheeled}, we discuss the notion divergence of a derivation of a free operadic algebra, and a new definition of a wheeled operad. In Section \ref{sec:stable-homology}, we prove our main results and discuss some of their applications. In Section \ref{sec:perspectives}, we discuss further consequences of our results and outline some perspectives and further directions. 

\section{Recollections}\label{sec:recoll}

In this section, we recall some standard background information on twisted associative algebras, operads, and operadic algebras, as well as universal multiplicative envelopes and K\"ahler differentials of operadic algebras; the reader is invited to consult the monographs \cite{MR2724388,MR1629341} and the original paper \cite{MR633783} for more information on the notion of species, the monograph \cite{MR2954392} for more details on operads and operadic algebras, and the monograph \cite{MR2494775} for more details on modules over operads, and a functorial view of universal multiplicative envelopes and K\"ahler differentials.\\

All vector spaces in this paper are defined over a field $\k$ of characteristic zero. For a finite set $I$ and a family of vector spaces $\{V_i\}_{i\in I}$, the unordered tensor product of these vector spaces is defined by 
 \[
\bigotimes_{i\in I} V_i:=\left(\bigoplus_{(i_1,\ldots,i_n) \text{ a total order on } I} V_{i_1}\otimes\cdots\otimes V_{i_n}\right)_{\Aut(I)}.
 \]
We denote by $S_n$ the symmetry group $\Aut(\{1,\ldots,n\})$.  
All chain complexes are graded homologically, with the differential of degree~$-1$. To handle suspensions of chain complexes, we introduce a formal symbol~$s$ of degree~$1$, and define, for a graded vector space~$V$, its suspension $sV$ as $\k s\otimes V$.

\subsection{Twisted associative algebras and operads}

A \emph{linear species} is a contravariant functor from the groupoid $\Fin$ of finite sets (the category whose objects are finite sets and whose morphisms are bijections) to the category of vector spaces. Sometimes, a ``skeletal definition'' is preferable: the category of linear species is equivalent to the category of symmetric sequences $\{\calS(n)\}_{n\ge 0}$, where each $\calS(n)$ is a right $S_n$-module, a morphism between the sequences $\calS_1$ and $\calS_2$ in this category is a sequence of $S_n$-equivariant maps $f_n\colon \calS_1(n)\to\calS_2(n)$. While this definition may seem more appealing, the functorial definition simplifies many definitions and proofs in a way that totally justifies the level of abstraction.\\

The \emph{derivative} $\partial(\calS)$ of a species $\calS$ is defined by the formula
 \[
\partial(\calS)(I):=\calS(I\sqcup\{\star\}).
 \]
The \emph{Cauchy product} of two linear species $\calS_1$ and $\calS_2$ is defined by the formula
 \[
(\calS_1\otimes\calS_2)(I):=\bigoplus_{I=I_1\sqcup I_2}\calS_1(I_1)\otimes\calS_2(I_2).
 \]
Equipped with this product, the category of linear species becomes a symmetric monoidal category, with the monoidal unit being the species $\mathbf{1}$ supported at the empty set with the value $\mathbf{1}(\varnothing)=\k$. A \emph{twisted associative algebra} is a monoid in this monoidal category.

The \emph{composition product} of two linear species $\calS_1$ and $\calS_2$ is compactly expressed via the Cauchy product as
 \[
\calS_1\circ\calS_2:=\bigoplus_{n\ge 0}\calS_1(\{1,\ldots,n\})\otimes_{\k S_n}\calS_2^{\otimes n}, 
 \]
that is, if one unwraps the definitions, 
 \[
(\calS_1\circ\calS_2)(I)
=\bigoplus_{n\ge 0}\calS_1(\{1,\ldots,n\})\otimes_{\k S_n}\left(\bigoplus_{I=I_1\sqcup \cdots\sqcup I_n}\calS_2(I_1)\otimes\cdots\otimes \calS_2(I_n)\right).
 \]
Equipped with this product, the category of linear species becomes a (very non-symmetric) monoidal category, with the monoidal unit being the species $\mathbbold{1}$ which vanishes on a finite set $I$ unless $|I|=1$, and whose value on $I=\{a\}$ is given by the one-dimensional vector space $\k a$. A \emph{(symmetric) operad} is a monoid in that monoidal category, that is a triple $(\calO,\gamma,\eta)$, where $\gamma\colon\calO\circ\calO\to\calO$ is the product and $\eta\colon\mathbbold{1}\to\calO$ is the unit, which satisfy the usual axioms of a monoid in a monoidal category \cite{MR0354798}. This defines algebraic (that is, $\k$-linear) operads, an object introduced in the paper of Artamonov \cite{MR0237408} under the slightly obscure name ``clones of multilinear operations'', shortly before the appearance of the term ``operad'' in its topological glory \cite{MR2177746,MR0420610}. 

It is also possible to encode operads using the \emph{partial} (or \emph{infinitesimal}) compositions 
 \[
\circ\star\colon\partial(\calO)\otimes\calO\to\calO.
 \]
Associativity of operadic composition translates into the requirement that these operations must satisfy the ``consecutive'' and the ``parallel'' axioms \cite{MR2954392}. Functorially, the consecutive axiom is given by the commutative diagram
 \[
 \xymatrix@M=6pt{
\partial(\calO)\otimes\partial(\calO)\otimes\calO\ar@{->}^{\id\otimes \circ_\star}[rr] \ar@{->}_{\mu\otimes \id}[d] & & \partial(\calO)\otimes\calO \ar@{->}^{\circ_\star}[d]  \\ 
\partial(\calO)\otimes\calO\ar@{->}_{\circ_\star}[rr]  & & \calO      
 } 
 \]
and the parallel axiom is given by the commutative diagram 
 \[
 \xymatrix@M=6pt{
\partial(\partial(\calO))\otimes\calO\otimes\calO\ar@{->}^{(\rho\otimes\id)(\id\otimes \sigma) }[rr] \ar@{->}_{\rho\otimes \id}[d] & & \partial(\calO)\otimes\calO \ar@{->}^{\circ_{\star_1}}[d]  \\ 
\partial(\calO)\otimes\calO\ar@{->}_{\circ_{\star_2}}[rr]  & & \calO      
 } .
 \]
Here structure morphisms
 \[
\mu\colon\partial(\calO)\otimes\partial(\calO)\to\partial(\calO) \quad\text{ and }\quad \rho\colon\partial(\partial(\calO))\otimes\calO\to\partial(\calO)
 \]
are obtained by applying $\partial$ to $\circ_\star$, and $\sigma\colon\calO\otimes\calO\cong\calO\otimes\calO$ is the symmetry isomorphism of the Cauchy product.

The prototypical example of an operad is given by the \emph{endomorphism operad} $\End_V$ of a vector space~$V$. It is given by $\End_V(I)=\Hom_{\k}(V^{\otimes I},V)$,
with the composition given by the usual composition of multilinear maps. In general, an operad corresponds to a $\k$-linear algebraic structure: for a class $\sfC$ of $\k$-linear algebras, one can define an operad $\calO_\sfC$ where $\calO_\sfC(I)$ consists of all multilinear operations with inputs indexed by $I$ that one can define in terms of the structure operations of $\sfC$. 

Unless otherwise specified, all operads $\calO$ we consider will be \emph{reduced} (that is, $\calO(\varnothing)=0$) and \emph{augmented} (that is, equipped with a map of operads $\epsilon\colon\calO\to\mathbbold{1}$ which is a left inverse to the monoid unit $\eta$); we shall denote by $\overline{\calO}$ the augmentation ideal of $\calO$ (the kernel of the augmentation). The reader may decide to only think of operads generated by finitely many binary operations, though our results hold in more general situations; in particular, allowing oneself unary operations is essential to recover the Loday--Quillen--Tsygan theorem mentioned in the introduction. 
A \emph{weight graded} operad is an operad $\calO$ whose underlying species is decomposed as a direct sum of species $\calO^{i}$, $i\in\mathbb{Z}$ so that applying any structure operation to elements of certain weights produces an element whose weight is the sum of the weights. One common example of a weight grading is given by assigning to an element of arity $n$ weight $n-1$. It is common to use the terminology ``weight grading'' to emphasize distinction from ``homological degree'': the latter also adds up under structure operations, but creates Koszul signs when exchanging elements of odd degrees, while weight gradings do not create any signs.  

Thinking of twisted associative algebras and of operads as monoids automatically leads to definitions of (left, right, and two-sided) modules using the corresponding monoidal structure. One can also use operadic modules to re-define twisted associative algebras: a twisted associative algebra is the same as a left module over the operad of associative algebras. Moreover, one can check that the Cauchy product makes the category of right modules over an operad $\calO$ into a symmetric monoidal category. 

 It is known \cite{MR2494775} that the derivative species $\partial(\calO)$ has two structures: of a twisted associative algebra and of a right $\calO$-module; in our notation, the former is given by the morphism $\mu$ and the latter by the morphism $\rho$. (Strictly speaking, $\rho$ gives an infinitesimal right module structure, but for unital operads, there is no difference between the notion of a right module and an infinitesimal right module.) Moreover, these two structures commute: $\partial(\calO)$ is a twisted associative algebra in the symmetric monoidal category of right $\calO$-modules, or, equivalently, an $\Ass$--$\calO$-bimodule.

\subsection{Monad of trees}
It is also important that we can view operads as algebras over a monad. Let us recall that viewpoint which will be useful for us when we generalize to wheeled operads later. We start by defining an endofunctor $\calT$ of the category of species by the following formula :
 \[
\calT(\calF)(I):=\bigoplus_{\substack{\tau \text{ a rooted tree, }\\ \mathsf{Leaves}(\tau)=I}}\bigotimes_{v\in\mathsf{Vertices}(\tau)}\calF(\mathrm{in}(v)),
 \] 
so in plain words, we decorate each internal vertex $v$ of $\tau$ with a label from $\calF$, taking the component of $\calF$ indexed by the set of incoming edges of $v$. For example, 
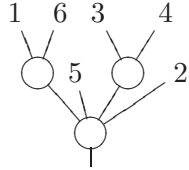
\[  \xygraph{ 
!{<0cm,0cm>;<1cm,0cm>:<0cm,0.8cm>::} 
!~-{@{-}@[|(2.5)]}
!{(0,0) }*+[o][F-]{\phantom{\alpha}}="a"
!{(0,-0.8) }*+{}{\phantom{\alpha}}="f"
!{(0.5,1) }*+[o][F-]{\phantom{\alpha}}="b"
!{(-0.2,1) }*+{5}="c"
!{(-0.7,1) }*+[o][F-]{\phantom{\alpha}}="d"
!{(1.2,1) }*+{2}="e"
!{(0.1,2) }*+{3}="g"
!{(1,2) }*+{4}="h"
!{(-1,2) }*+{1}="k"
!{(-0.4,2) }*+{6}="l"
"k" - "d"
"l" - "d"
"g" - "b"
"h" - "b"
"b" - "a" 
"c" - "a"
"e" - "a"
"d" - "a"
"a" - "f"
}
\]
is a typical rooted tree with set of leaves $\{1,2,\ldots,6\}$, and the summand corresponding to that tree is $\calF(\{1,6\})\otimes\calF(\{3,4\})\otimes\calF(\{2,5,\star_1,\star_2\})$, where $\star_1$ and $\star_2$ correspond to the points of grafting of the two subtrees.
It is easy to see that $\calT$ is in fact a monad: the natural transformation $\calT\calT\to\calT$ comes from the fact that a rooted tree with tiny rooted trees inserted in its internal  vertices can be viewed as a rooted tree if we forget this nested structure of insertions, and the natural transformation $\mathsf{1}\to\calT$ comes from considering rooted trees with one internal vertex (``corollas''). We call this monad the \emph{monad of rooted trees}. An operad is an algebra over $\calT$, that is a species $\calO$ equipped with a structure map $\calT(\calO)\to \calO$. If we keep the intuition of multilinear operations, the three viewpoints we discussed are as follows: the monoid viewpoint corresponds to composing operations along two-level rooted trees, the partial composition viewpoint corresponds to composing operations along rooted trees with two internal vertices, and the monad viewpoint corresponds to composing operations along arbitrary rooted trees.

\subsection{PROPs and coPROPs}
We shall also use PROPs, which are versions of operads with several inputs and several outputs (though introduced before operads by Adams and Mac Lane \cite{MR0171826}), so that one should think of bi-species: functors from $\Fin^{op}\times\Fin$ to the category of vector spaces. However, the definition can be given more quickly: a PROP is a  $\k$-linear symmetric monoidal category $\calP$ whose objects are natural numbers, and the monoidal structure is addition. Then the symmetry isomorphisms act on $n=1+1+\cdots+1$, and the hom-spaces $\calP(m,n)$ naturally have a left action of $S_n$ and a commuting right action of $S_m$, which is essentially the same as a bi-species. Moreover, there are 
``vertical'' compositions 
 \[
 \calP(m,n)\otimes\calP(k,m)\to\calP(k,n)
 \] 
that are just compositions of morphisms in the category $\calP$ and ``horizontal'' compositions 
 \[
\calP(m,n)\otimes\calP(m',n')\to\calP(m+m',n+n')
 \] 
that correspond to the tensor product of morphisms. These are associative and satisfy the ``interchange law'' between them. 

To each operad $\calO$, one can associate a PROP $\calP(\calO)$ whose components are given by the formula
 \[
(\calP(\calO))(m,n):=\calO^{\otimes n}(m),
 \]
where $\calO^{\otimes n}$ is the $n$-th tensor power of the species $\calO$ with respect to the Cauchy product, the action of $S_n$ comes from the symmetry isomorphisms of the Cauchy product, the vertical composition comes from the composition in $\calO$, and the horizontal composition is the product in the tensor algebra $\calO^{\otimes\bullet}$. This PROP is often referred to as the PROP completion of the operad $\calO$; in fact, it appears in \cite{MR0236922} under the name ``category of operators in standard form'' as an easily identifiable precursor of the notion of an operad. 

There is also a notion of a cooperad, where all arrows are reversed, and the notion of a coPROP; the coPROP completion of a cooperad $\calC$ is denoted $\calP^c(\calC)$. 

\subsection{Wheeled bar construction and Koszulness}

Given an augmented operad $\calO$, its bar construction $\mathsf{B}(\calO)$ is $\calT(s\overline{\calO})$ equipped with the differential given by the sum of all ways of collapsing edges of the tree and using the operad structure maps on the labels of the vertices that are identified by such collapses. Moreover, one can talk about cooperads by reversing all arrows in the respective definitions; in that language, the bar construction of an operad $\calO$ is the cofree conilpotent cooperad on $s\overline{\calO}$ equipped with the unique coderivation extending the maps induced by the structure maps of the operad $\calO$. 

The bar construction can be extended to dg operads, making it one half of a Quillen adjunction between augmented differential graded operads and conilpotent coaugmented differential graded cooperads; the other half of that adjunction is given by the cobar construction for cooperads, which is defined analogously. 

A weight graded operad with all generators of weight one is said to be \emph{Koszul} if the homology of its bar construction is concentrated on the diagonal (weight equal to homological degree).

\subsection{Algebras over operads}

A vector space $V$ can be considered naturally as a species supported at the empty set. If $\calS$ is another species, we have
 \[
\calS\circ V:=\bigoplus_{n\ge 0}\calS(\{1,\ldots,n\})\otimes_{\k S_n} V^{\otimes n}. 
 \]
Classically, the right hand side of this formula is called the Schur functor associated to $\calS$ and is denoted $\calS(V)$. 

The main reason to study operads is because of their algebras. For an operad~$\calO$, an algebra over $\calO$ is a vector space $A$ together with a structure of a left $\calO$-module on the species $A$. One can also check that a structure of an $\calO$-algebra on $A$ is the same as a morphism of operads $\calO\to\End_A$; in this sense, $\calO$-algebras are representations of  $\calO$. In this paper, we shall only consider free operadic algebras; the free $\calO$-algebra over $\calO$ generated by a vector space $V$ is the vector space $\calO(V)=\calO\circ V$, equipped with the obvious left module structure.

\subsection{Derivations}
Recall that a derivation of an $\calO$-algebra $A$ is a linear map $D\colon A\to A$ satisfying 
 \[
D(\mu(a_1,\ldots,a_n))=\sum_i \mu(a_1,\ldots,D(a_i),\ldots,a_n)
 \]
for every structure map $\mu$ of the operad $\calO$. Informally, derivations are ``infinitesimal automorphisms''; one way to make some sense of this statement is to think of the automorphism $\mathrm{Id}+\epsilon D$ of $A[\epsilon]/(\epsilon^2)$. Note that for an operation $\mu$ of arity one, the derivation property means that $D$ commutes with $\mu$, so a derivation of an $\calO$-algebra $A$ is in particular an endomorphism of the $\calO(1)$-module structure on $A$. The set of derivations of $A$, denoted $\Der(A)$, has an obvious structure of a vector space, and moreover of a Lie algebra with respect to the Lie bracket $[D,D']=D\circ D'-D'\circ D$. 

The Lie algebra $\Der(\calO(V))$ has a Lie subalgebra $\mathfrak{gl}(V)$ consisting of derivations for which the image of every element of $V$ is in $V$. For an augmented operad, one can use the augmentation to define a surjective Lie algebra morphism $\Der(\calO(V))\twoheadrightarrow \mathfrak{gl}(V)$; we shall denote its kernel $\Der^+(\calO(V))$. It is also worth noting that the Lie algebra $\Der(\calO(V))$ has an obvious integer grading: a derivation has weight $w$ if it sends each generator $v\in V$ into a combination of elements obtained by applying elements of $\calO(w+1)$ to elements of $V$; the weight is non-negative for a reduced operad, and may assume value $-1$ otherwise. Moreover, if $\calO$ is connected (all unary operations are proportional to the identity), $\Der^+(\calO(V))$ consists of derivations of positive weight, but in general it also contains some elements of weight zero. 

It is clear that for a free algebra $\calO(V)$ any linear map $V\to\calO(V)$ extends to a unique $\calO$-algebra derivation; in other words, we have a canonical vector space isomorphism 
 \[
\Der(\calO(V))\cong \Hom(V,\calO(V)),
 \] 
and we shall denote by $\imath_D$ the element of $\Hom(V,\calO(V))$ corresponding to $D$. This isomorphism allows us to see that for free algebras, the Lie bracket of derivations arises from a pre-Lie algebra structure on $\Der(\calO(V))$: we can define, for two derivations $D$, $D'$, the pre-Lie product $D\triangleleft D'$ to be given by
 \[
\imath_{D\triangleleft D'}(v):=D(\imath_{D'}(v)).
 \]
\begin{remark}
Sometimes, it is useful that one can view the pre-Lie algebra structure on $\Der(\calO(V))$ as follows. There is an operad with the $n$-th component $\Hom(V,V^{\otimes n})\otimes\calO(I)$, the Hadamard product of the ``coendomorphism operad'' and $\calO$; this operad has an associated pre-Lie algebra \cite{MR2954392}, and we are considering the subalgebra of that pre-Lie algebra made of invariants of actions of symmetric groups on the components of the operad.
\end{remark} 

\subsection{Universal multiplicative envelopes and K\"ahler differentials}
Let $A$ be an $\calO$-algebra. Informally, the \emph{universal multiplicative enveloping algebra} $U_{\calO}(A)$ encodes all ``actions of elements of $A$ on a general (bi)module''. Formally, we define it via a relative composite product construction \cite{MR2494775}:
 \[
U_{\calO}(A)\cong \partial(\calO)\circ_{\calO} A,
 \]
with the associative algebra structure induced by the left $\Ass$-module structure of the $\Ass$--$\calO$-bimodule $\partial(\calO)$. For the universal multiplicative enveloping algebra of a free $\calO$-algebra, we have an associative algebra isomorphism
 \[
U_{\calO}(\calO(V))\cong \partial(\calO)\circ_{\calO} \calO(V)\cong \partial(\calO)(V).
 \] 

\begin{example}
Let $\Ass$ denote the operad of associative algebras, and $\uAss$ denote the operad of unital associative algebras. The species $\partial(\Ass)$ is naturally identified with the Cauchy product $\uAss\otimes\uAss$. Indeed, if we have an element $a_{i_1}\cdots a_{i_n}\in\Ass(I\sqcup\{\star\})$, then the position of $a_\star $ in this product splits it into two possibly empty parts which is precisely the Cauchy product decomposition. The first factor corresponds to left multiplications, and the second factor to the right ones. In particular, if we consider the free non-unital associtive algebra $\overline{T}(V)$, its universal multiplicative envelope is isomorphic to $T(V)\otimes T(V)$.
\end{example}

Generalizing the classical definition for commutative algebras, one may define the $U_{\calO}(A)$-module of K\"ahler differentials $\Omega^1_A$ for any given operad $\calO$ and any given $\calO$-algebra $A$ \cite[Sec.~4.4]{MR2494775}. 
The intrinsic definition states that $\Omega^1_{\calO}(A)$ is the $U_{\calO}(A)$-module that represents the functor of derivations $\mathrm{Der}(A,E)$ with values in a $U_{\calO}(A)$-module $E$.  
Explicitly this module is spanned by formal expressions $p(a_1,\ldots,d a_i,\ldots,a_m)$, where $p\in\calO(m)$, and $a_1, \ldots, a_m\in A$, modulo the relations
\begin{multline*}
p(a_1,\ldots,q(a_i,\ldots,a_{i+n-1}),\ldots, d a_j,\ldots, a_{m+n-1})=\\
(p\circ_i q)(a_1,\ldots,a_i,\ldots,a_{i+n-1},\ldots, d a_j,\ldots, a_{m+n-1}),
\end{multline*}
\begin{multline*}
p(a_1,\ldots,d q(a_i,\ldots,a_{i+n-1}),\ldots, a_{m+n-1})=\\
\sum_{j=i}^{i+n-1}(p\circ_i q)(a_1,\ldots,a_i,\ldots,d a_j,\ldots,a_{i+n-1},\ldots, a_{m+n-1}).
\end{multline*}
We have the ``universal derivation''  
 \[
\d\colon A\to \Omega^1_{\calO}(A)
 \]
sending $a\in A$ to $\id(da)$. 

It is well known that for a free algebra $\calO(V)$, we have a canonical left $U_{\calO}(\calO(V))$-module isomorphism
 \[
\Omega^1_{\calO}(\calO(V))\cong U_{\calO}(\calO(V))\otimes V,
 \]
uniquely determined by identifying $\d(v)$ with $1\otimes v$ for all $v\in V$. 
Since we have an associative algebra isomorphism $U_{\calO}(\calO(V))\cong \partial(\calO)(V)$, and since every derivation of a free algebra is defined on generators, the universal derivation of the free algebra can be regarded as a map
 \[
V\to \partial(\calO)(V)\otimes V.  
 \]

\begin{example}
We continue considering the operad $\Ass$ of associative algebras. In this particular case, our general formulas imply that
 \[
\Omega^1_{\Ass}(\Ass(V))\cong U_{\Ass}(\Ass(V))\otimes V\cong T(V)\otimes T(V)\otimes V.
 \]
 Let us give an illustration of how that isomorphism is implemented. Let $V=\mathrm{Vect}(x,y)$, and consider $xyx\in \overline{T}(V))$. We have 
\begin{multline*}
\d(xyx)=\d(x)yx+x\d(y)x+xy\d(x)=\\ r_{yx}\d(x)+r_xl_x\d(y)+l_{xy}\d(x)=
(l_{xy}+r_{yx})\d(x)+r_xl_x\d(y).
\end{multline*}
 Here $l_{a}$ denotes $a\otimes 1\in T(V)\otimes T(V)$, and $r_a$ denotes $1\otimes a\in T(V)\otimes T(V)$.
\end{example}

\subsection{Cyclic homology}

Let us briefly recall some key results concerning the cyclic homology of an associative algebra over a field of zero characteristic; we refer the reader to the monographs of Loday \cite{MR1600246} and Feigin and Tsygan \cite{MR0923136} for details. 

Given an associative algebra $A$, one can consider a chain complex structure on the (non-counital) cofree conilpotent coalgebra $\overline{T}^c(sA)$ where the differential $d$ is the unique coderivation extending the map 
 \[
\overline{T}^c(sA)\twoheadrightarrow sA\otimes sA\to sA
 \] 
which is obtained as the composition of the obvious projection with the product in $A$. That chain complex is often referred to as the bar construction of $A$; we however choose to keep the notation of \cite{MR2954392} and define the bar construction for augmented algebras, applying the above recipe to the augmentation ideal. The cyclic permutation $\sigma_n=(1\ 2\ \ldots\ n)$ acts on the component $(sA)^{\otimes n}$ of the coalgebra $\overline{T}^c(sA)$, and the images of $\id-\sigma_n$ assemble into a subcomplex of $\overline{T}^c(sA)$. The underlying object of the quotient complex is the space of cyclic words $\Cyc(sA)$, and cyclic homology of $A$ may be defined as the homology of the degree one shift of that complex:
 \[
HC_\bullet(A):=H_\bullet(s^{-1}\Cyc(sA),d).
 \]
In particular, $HC_0(A)=A/[A,A]$, and another definition of cyclic homology interprets it as higher nonabelian derived functors of the commutator quotient 
 \[
|A|:=A/[A,A].
 \]
Concretely, if $R_\bullet=(R,d_R)\simeq A$ is a quasi-free resolution of $A$, i.e., a differential graded associative algebra whose underlying associative algebra $R$ is free and whose homology is isomorphic to $A$, then  
 \[
HC_\bullet(A)\cong H_\bullet(R_\bullet/[R_\bullet,R_\bullet]).
 \]
One of our results, Theorem \ref{th:calchom} below, uses cyclic homology of twisted associative algebras; over a field of characteristic zero, all the results about 
cyclic homology of associative algebras extend \emph{mutatis mutandis} to twisted associative algebras, if one replaces the tensor product of vector spaces by the Cauchy product. Alternatively, for a twisted associative algebra $\calA$ and any vector space~$V$, the vector space $\calA(V)$ has a canonical structure of an associative algebra induced by the twisted associative algebra structure of $\calA$, and one can view cyclic homology calculations for $\calA$ as cyclic homology calculations for all $\calA(V)$ at the same time, keeping the answer functorial in $V$. 

\section{Divergence of derivations and wheeled operads}\label{sec:div-wheeled}

In this section, we discuss two notions that will be of importance for our main results: divergence of derivations of free operadic algebras and wheeled operads. While formally no material of this section is new (an equivalent definition of divergence is due to Powell \cite{Powell21}, and an equivalent definition of a wheeled operad is given in the foundational paper of Markl, Merkulov, and Shadrin \cite{MR2483835}), our viewpoints on divergence and on wheeled operads are new, and we believe that these viewpoints fill a gap in the existing literature, clarifying the meaning of those objects in a substantial way.

\subsection{Divergence of derivations}
We shall now introduce the notion of divergence of a derivation of $\calO(V)$. This is equivalent to the definition recently given by Powell \cite{Powell21}. Though in \cite[Appendix B]{Powell21} a viewpoint using universal multiplicative enveloping algebras is indicated, it is not the approach that is pursued in that paper; we believe that our viewpoint has the advantage of hinting at the role of wheeled operads in this context.  

Let $D$ be a derivation of $\calO(V)$. The \emph{divergence} of the derivation $D$, denoted by $\Div(D)$, is defined
as follows. We compose the map $\imath_D\in\Hom(V,\calO(V))$ with the universal derivation, obtaining a map
 \[
\mathrm{d}\circ \imath_D\colon V\to \partial(\calO)(V)\otimes V, 
 \]
which we identify, using the hom-tensor adjunction and symmetry isomorphisms of tensor products, with an element $\jmath(\mathrm{d}\circ \imath_D)$ of $V^*\otimes V\otimes \partial(\calO)(V)$. This latter is an associative algebra $A=\Hom(V,V)\otimes\partial(\calO)(V)$, and we can take the universal trace on it, which is the canonical projection to the commutator quotient~$|A|$. It is well known that the commutator quotient of a tensor product of two associative algebras is isomorphic to the tensor product of commutator quotients and $|\Hom(V,V)|=\k$, so 
 \[
|\Hom(V,V)\otimes\partial(\calO)(V)|\cong |\partial(\calO)(V)|. 
 \]
For a derivation $D\in\Der(\calO(V))$, its \emph{divergence} is the element 
 \[
\Div(D)\in |\partial(\calO)(V)|
 \] 
given by the universal trace of the element $\jmath(\mathrm{d}\circ \imath_D)\in \Hom(V,V)\otimes\partial(\calO)(V)$.

\begin{remark}
If one chooses a basis $x_1,\ldots,x_n$ of $V$, each derivation $D$ of $\calO(V)$ is uniquely determined by its components $D_i=D(x_i)$, and, once we define ``partial derivatives'' $\frac{\partial f}{\partial x_i}\in\calO(V)$ of $f\in\calO(V)$ by the formula 
 \[
\d(f)=:\sum_{i=1}^n\frac{\partial f}{\partial x_i}\d(x_i),
 \]
the divergence of $D$ is computed by writing the ``classical'' formula 
 \[
\sum_{i=1}^n\frac{\partial{D_i}}{\partial x_i},
 \]
and projecting the result to the commutator quotient $|\partial(\calO)(V)|$.
\end{remark}

Let us consider a simple yet useful example. 

\begin{example}\label{ex:divforalg}
Let $A$ be a unital associative algebra, and consider it as an operad $\calO$ that only has operations in arity one given by $A$. In this case, the free $\calO$-algebra $\calO(V)$ generated by a vector space $V$ is just the free left module $A\otimes V$, and a derivation of $\calO(V)$ is just an $A$-module endomorphism. Since we have isomorphisms of associative algebras  
 \[
\End_A(A\otimes V)\cong \End_A(A)\otimes \End(V)\cong A^{\mathrm{op}}\otimes\End(V),
 \]
for $V=\k^n$, the Lie algebra $\Der(\calO(V))$ is the Lie algebra $\mathfrak{gl}_n(A^{\mathrm{op}})\cong \mathfrak{gl}_n(A)$.

Moreover, the species derivative $\partial(\calO)$ is supported at the empty set and coincides with $A$, and we have an associative algebra isomorphism $\partial(\calO)(V)\cong A$. Finally, the module of K\"ahler differentials $\Omega^1_{\calO}(\calO(V))$ is isomorphic to $A\otimes V$, and the universal derivation 
 \[
\d\colon \calO(V)\to \Omega^1_{\calO}(\calO(V))
 \]
is in this case the identity map of $A\otimes V$. From that, it immediately follows that the divergence of derivations of $\calO(\k^n)$ is given by the composition 
 \[
\mathfrak{gl}_n(A)\to A\to |A|,
 \] 
where the first map computes the trace of the matrix, and the second map is the canonical projection onto the commutator quotient.
\end{example}

Let us give another example that is a bit more concrete.

\begin{example}\label{ex:bergman}
We continue considering the operad $\Ass$ of associative algebras. Let $V=\mathrm{Vect}(x,y)$, and consider the derivation $D\in \Der(\Ass(V))$ that sends the generator $x$ to $[x,y]^2$ and the generator $y$ to zero. (It is the tangent derivation of the Bergman automorphism of the algebra generated by two generic $2\times 2$-matrices \cite{bergman1979}.) To compute the divergence of $D$, we should compute the projection of
 \[
\frac{\partial D(x)}{\partial x}+\frac{\partial D(y)}{\partial y}
 \]
to the commutator quotient. Since
 \[
\partial D(x)=[x,y]^2=xyxy-xyyx-yxxy+yxyx, 
 \] 
we have   
\begin{multline*}
\d([x,y]^2)=\d(x)yxy+x\d(y)xy+xy\d(x)y+xyx\d(y)\\
-\d(x)yyx-x\d(y)yx-xy\d(y)x-xyy\d(x)\\
-\d(y)xxy-y\d(x)xy-yx\d(x)y-yxx\d(y)\\
+\d(y)xyx+y\d(x)yx+yx\d(y)x+yxy\d(x),
\end{multline*}
implying that
\begin{multline*}
\d([x,y]^2)=
(r_{yxy}+l_{xy}r_y-r_{yyx}-l_{xyy}-l_yr_{xy}-l_{yx}r_y+l_yr_{yx}+l_{yxy})\d(x)\\
+(r_{xy}l_x+l_{xyx}-r_{yx}l_x-l_{xy}r_x-r_{xxy}-l_{yxx}+r_{xyx}+l_{yx}r_x)\d(y).
\end{multline*}
Additionally, we have $D(y)=0$, so we obtain
\begin{multline*}
\frac{\partial D(x)}{\partial x}+\frac{\partial D(y)}{\partial y}=r_{yxy}+l_{xy}r_y-r_{yyx}-l_{xyy}-l_yr_{xy}-l_{yx}r_y+l_yr_{yx}+l_{yxy}=\\
r_yr_xr_y+l_xl_yr_y-r_xr_yr_y-l_xl_yl_y-l_yr_yr_x-l_yl_xr_y+l_yr_xr_y+l_yl_xl_y,
\end{multline*}
which manifestly vanishes in the commutator quotient of the universal enveloping algebra. Thus, $\Div(D)=0$.
\end{example}

A fundamental property of divergence is given by the fact that it is a cocycle for the Lie algebra of derivations. To make it precise, let us note that each derivation of an associative algebra $A$ preserves $[A,A]$, and thus induces an endomorphism of the commutator quotient~$|A|$. Since a derivation $D$ of $\calO(V)$ induces a unique derivation of the associative algebra $\partial(\calO)(V)\cong U_{\calO}(\calO(V))$, it follows that $D$ induces an endomorphism of $|\partial(\calO)(V)|$, which we shall denote by $D_*$. 

\begin{proposition}\label{prop:divcocycle}
Divergence is a $1$-cocycle of $\Der(\calO(V))$ with values in $|\partial(\calO)(V)|$: for two derivations $D,D'$ of $\calO(V)$, we have 
 \[
\Div([D,D'])=D_*(\Div(D'))-D'_*(\Div(D)).
 \]
\end{proposition}

\begin{proof}
Our definitions immediately imply
 \[
\jmath(\mathrm{d}\circ\imath_{D\triangleleft D'})=D_*(\jmath(\mathrm{d}\circ\imath_{ D'}))+\jmath(\mathrm{d}\circ\imath_{ D'})\jmath(\mathrm{d}\circ\imath_{ D}), 
 \]
implying that
\begin{multline*}
\jmath(\mathrm{d}\circ\imath_{[D,D']})=\jmath(\mathrm{d}\circ\imath_{D\triangleleft D'-D'\triangleleft D})=\\ D_*(\jmath(\mathrm{d}\circ\imath_{ D'}))-D'_*(\jmath(\mathrm{d}\circ\imath_{ D}))+[\jmath(\mathrm{d}\circ\imath_{ D'}),\jmath(\mathrm{d}\circ\imath_{ D})], 
\end{multline*}
so in the commutator quotient we have
\[
\Div([D,D'])=D_*(\Div(D'))-D'_*(\Div(D)),
 \]
as required. 
\end{proof}

As an immediate consequence, all derivations with constant divergence (that is, proportional to the coset of the unit element in the commutator quotient) form a Lie subalgebra of the Lie algebra $\Der(\calO(V))$, and all derivations with zero divergence form an ideal of codimension one of that latter subalgebra. We shall denote the former subalgebra $\SDer^\wedge(\calO(V))$, and the latter one $\SDer(\calO(V))$. For an augmented operad, the augmentation of $\calO$ induces surjective Lie algebra morphisms $\SDer^\wedge(\calO(V))\twoheadrightarrow\mathfrak{gl}(V)$ and $\SDer(\calO(V))\twoheadrightarrow\mathfrak{sl}(V)$ with the same kernel which we shall denote $\SDer^+(\calO(V))$.

\begin{example}\label{ex:ULie}
It is well known that the Lie algebra $\Der^+(\Lie(V))$ has generators of all weights \cite{kassabov2003automorphism,MR3047471}. It is possible to give an easy explanation of that fact using the notion of divergence. First, we recall that the universal multiplicative enveloping algebra of a Lie algebra is known to coincide with the left adjoint of the functor from associative algebras to Lie algebras that retains the Lie bracket only, and hence universal enveloping algebras of free Lie algebras are free, so that we have an isomorphism of twisted associative algebras $\partial(\Lie)\cong\uAss$. From the cocycle property 
 \[
\Div([D,D'])=D_*(\Div(D'))-D'_*(\Div(D))
 \]
and the fact that for each derivation $D$ of positive weight of the free Lie algebra, the derivation $D_*$ sends 
 \[
V\subset U_{\Lie}(\Lie(V))\cong \uAss(V)\cong T(V)
 \]
to a linear combination of nontrivial Lie monomials, which, in particular, vanishes in the commutator quotient, it follows that every commutator of two derivations of positive weight has zero divergence, and so we need generators in each weight to account for derivations with nonzero divergence.
\end{example}

\subsection{Wheeled operads}\label{sec:wheeledop}

The notion of a wheeled operad introduced in \cite{MR2483835} is designed to encode $\calP$-algebras for which, together with structure operations, we also keep track of ``universal traces'' of elements of the universal multiplicative envelope. In \cite{MR2483835}, the notion of a wheeled operad is obtained as a specialization of the notion of a wheeled properad, which results in some definitions being less straightforward than they could be. In this section, we present these definitions in the way which will be more useful for our purposes. (We also choose a different convention for gradings in the bar construction to make everything agree with the conventions of \cite{MR2954392} in the operadic case.)

We shall now consider two-coloured linear species, that of contravariant functors $\calF$ from $\Fin$ to the category of vector spaces decomposed into a direct sum of two subspaces. We shall write $\calF(I)=:\calF_\o(I)\oplus\calF_\w(I)$, and think of these summands as ``the operadic part'' and ``the wheeled part'' of $\calF(I)$, or subspaces of ``operations'' with one output and no outputs, respectively. 

A \emph{wheeled operad} is a two-coloured linear species $\calU=\calU_\o\oplus\calU_\w$, where:
\begin{itemize}
\item[-] $\calU_\o$ is an operad, 
\item[-] $\calU_\w$ is a right $\calU_\o$-module, 
\item[-] $\partial(\calU_\o)$ is given a \emph{trace map} \[\tr\colon\partial(\calU_\o)\to\calU_\w,\] which is a morphism of right $\calU_\o$-modules and 
vanishes on the commutators in the twisted associative algebra $\partial(\calU_\o)$. 
\end{itemize}
(Note that commutators in $\partial(\calO)$ are commutators in the graded sense: they are defined using the left $\Ass$-module structure, and so appropriate Koszul signs appear automatically.) A morphism of wheeled operads is a map of two-coloured linear species that agrees with all the three structures above. 

To each operad $\calO$, one can associate canonically two different wheeled operads. The first option is to consider $\calO$ as a trivial wheeled operad $\calO\oplus0$ postulating that all trace maps are equal to zero. The second option is to consider the \emph{wheeled completion} $\calO^{\circlearrowleft}$ with
 \[
\calO^{\circlearrowleft}_\o:=\calO, \quad \calO^{\circlearrowleft}_\w:=|\partial(\calO)|,
 \] 
where the trace map 
$\partial(\calO)\to |\partial(\calO)|$ is the canonical projection. This notion was first introduced (some years before the introduction of wheeled operads) by Kapranov in \cite[Sec.~3.3]{MR1671737}, where he interpreted what we identify as $|\partial(\calO)|$ as the ``module of natural forms'' on $\calO$-algebras. For a given operad $\calO$, all wheeled operads $\calU$ with $\calU_\o=\calO$ form a category, and the wheeled operads $\calO$ and $\calO^{\circlearrowleft}$ are, respectively, the terminal and the initial objects of that category. 

The unit wheeled operad is the wheeled completion $\mathbbold{1}^{\circlearrowleft}$ of the trivial operad~$\mathbbold{1}$; concretely, we have $\mathbbold{1}^{\circlearrowleft}_\o=\mathbbold{1}^{\circlearrowleft}_\w=\k$, and all the structure maps of $\mathbbold{1}^{\circlearrowleft}$ coincide (once we identify all sources and targets of these maps with $\k$ using the monoidal unit property), with the identity map of $\k$. The unit map $\eta\colon\mathbbold{1}\to\calU_\o$ induces the corresponding unit map of wheeled operads $\eta^{\circlearrowleft}\colon\mathbbold{1}^{\circlearrowleft}\to\calU$. We shall only consider \emph{augmented} wheeled operads $\calU$ equipped with a map $\epsilon^{\circlearrowleft}\colon\calU\to \mathbbold{1}^{\circlearrowleft}$ which is a left inverse of the augmentation, and we denote by $\overline{\calU}=\overline{\calU}_\o\oplus\overline{\calU}_\w$ the kernel of the augmentation.    

Just as the prototypical example of an operad is given by the endomorphism operad, one may consider the \emph{wheeled endomorphism operad} $\WEnd_V$ of a finite-dimensional vector space~$V$. It is given by 
 \[
\WEnd_V=\Hom_{\k}(V^{\otimes I},V)\oplus\Hom_{\k}(V^{\otimes I},\k),
 \]
where the operad and the right module structures are obvious, and where the trace map 
 \[
\tr\colon\Hom(V^{\otimes (I\sqcup\{\star\})},V)\to \Hom(V^{\otimes I},\k)
 \]
is obtained as the composite
 \[
\Hom(V^{\otimes (I\sqcup\{\star\})},V)\to \Hom(V^{\otimes I},\k)
 \]
of the contraction of a multilinear map with respect to the position $\star$. (The trace map vanishes on commutators of $\partial(\End_V)$ for the same reasons as the usual trace of endomorphisms does.) 

Let us also give a monadic definition of wheeled operads. Recall that we have an endofunctor of the category of species given by the monad of rooted trees. Existence of the root gives each edge of the rooted tree a canonical direction towards the root. Let us consider a slightly more general class of (isomorphism classes of) graphs with directed edges by only postulating that each internal vertex has at most one output edge, and not imposing any other restrictions. We shall refer to vertices with one output edges as corollas, and to vertices with no output edges as cul-de-sac vertices. 

\begin{lemma}[{\cite[Lemma 1.4]{MR0725416}}]\label{lm:fuchs}
If such a graph $\Gamma$ is connected, it has at most one outgoing edge.
\end{lemma}

\begin{proof}
If there are at least two outgoing edges, we may consider a (non-directed) path connecting them in $\Gamma$; one of the vertices of the path will have two outgoing edges, which is a contradiction.
\end{proof}

More precisely, there are three possible types of connected graphs that we allow: rooted trees (with one output edge), cul-de-sac trees (without output edges), and wheels (without output edges), as in the following picture: 
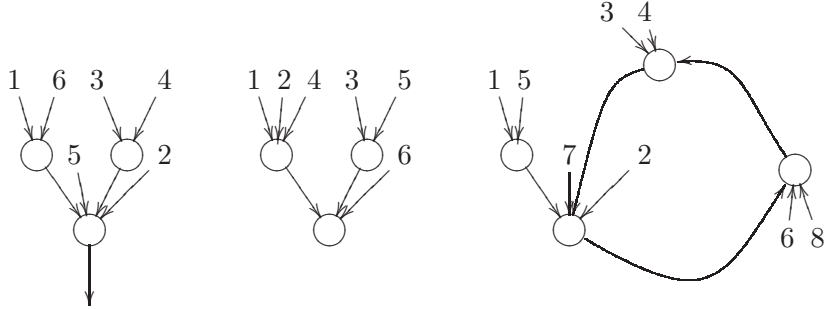
\[  \xygraph{ 
!{<0cm,0cm>;<1cm,0cm>:<0cm,1cm>::} 
!~-{@{-}@[|(2.5)]}
!{(0,0) }*+[o][F-]{\phantom{\alpha}}="a"
!{(0.5,1) }*+[o][F-]{\phantom{\alpha}}="b"
!{(-0.2,1) }*+{5}="c"
!{(-0.7,1) }*+[o][F-]{\phantom{\alpha}}="d"
!{(1,1) }*+{2}="e"
!{(0,-1) }="f"
!{(0.1,2) }*+{3}="g"
!{(1,2) }*+{4}="h"
!{(-1,2) }*+{1}="k"
!{(-0.4,2) }*+{6}="l"
"b" : "a" 
"c" : "a"
"e" : "a"
"d" : "a"
"a" : "f"
"g" : "b"
"h" : "b"
"k" : "d"
"l" : "d"
}\qquad
\xygraph{ 
!{<0cm,0cm>;<1cm,0cm>:<0cm,1cm>::} 
!~-{@{-}@[|(2.5)]}
!{(0,0) }*+[o][F-]{\phantom{\alpha}}="a"
!{(0.5,1) }*+[o][F-]{\phantom{\alpha}}="b"
!{(-0.2,2) }*+{4}="c"
!{(-0.7,1) }*+[o][F-]{\phantom{\alpha}}="d"
!{(1,1) }*+{6}="e"
!{(0.3,2) }*+{3}="g"
!{(1,2) }*+{5}="h"
!{(-1,2) }*+{1}="k"
!{(-0.6,2) }*+{2}="l"
"b" : "a" 
"c" : "d"
"e" : "a"
"d" : "a"
"g" : "b"
"h" : "b"
"k" : "d"
"l" : "d"
}
\qquad  \xygraph{ 
!{<0cm,0cm>;<1cm,0cm>:<0cm,1cm>::} 
!~-{@{-}@[|(2.5)]}
!{(0,0) }*+[o][F-]{\phantom{\alpha}}="a"
!{(1.2,2.2) }*+[o][F-]{\phantom{\alpha}}="b"
!{(-0.7,1) }*+[o][F-]{\phantom{\alpha}}="d"
!{(0,1) }*+{7}="f"
!{(1,1) }*+{2}="e"
!{(0.5,2.9) }*+{3}="h"
!{(1,2.9) }*+{4}="m"
!{(-1,2) }*+{1}="k"
!{(-0.6,2) }*+{5}="l"
!{(3,0.8) }*+[o][F-]{\phantom{\alpha}}="g"
!{(2.9,-0.1) }*+{6}="u"
!{(3.3,-0.1) }*+{8}="v"
"e" : "a"
"d" : "a"
"f" : "a"
"h" : "b"
"m" : "b"
"k" : "d"
"l" : "d"
"u" : "g"
"v" : "g"
"a" : @`{p+(-1,-2),c+(0,0.5)} "g"
"g" : @`{p+(1,-0.3),c+(0.2,0.5)} "b"
"b" : @`{p+(0.5,2),c+(0,0.1)} "a" 
}
\]

Using our three kinds of graphs, we define an endofunctor $\calT^{\circlearrowleft}$ of the category of two-coloured linear species by setting $\calT^{\circlearrowleft}(\calF)_\o(I)=\calT(\calF_\o)(I)$ and 
\begin{multline*}
\calT^{\circlearrowleft}(\calF)_\w(I):=\\ \bigoplus_{\substack{\tau \text{ a sink, }\\ \mathsf{Leaves}(\tau)=I}}\bigotimes_{v\in\mathsf{Vertices}(\tau)}\calF_{|\mathrm{out}(v)|}(\mathrm{in}(v))\oplus \bigoplus_{\substack{\tau \text{ a wheel, }\\ \mathsf{Leaves}(\tau)=I}}\bigotimes_{v\in\mathsf{Vertices}(\tau)}\calF_\o(\mathrm{in}(v)).
\end{multline*} 
This endofunctor can be given a monad structure that is completely analogous to the rooted tree monad, via insertions of graphs into graphs for the monad product, and via considering one-vertex trees for the monad unit. We call this monad the monad of \emph{rooted trees and wheels}.  A wheeled operad is an algebra over that monad. The trace map corresponds to creating a wheel out of a rooted tree (gluing the output edge to one of the input edges), and vanishing on the commutators corresponds to the fact that the following two wheel graphs are the same:
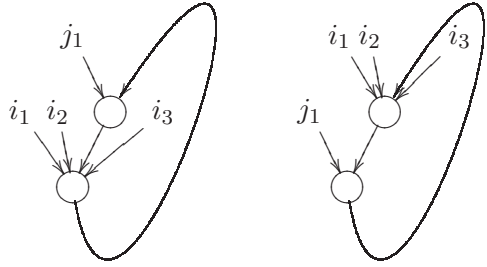
\[  \xygraph{ 
!{<0cm,0cm>;<1cm,0cm>:<0cm,1cm>::} 
!{(0,0) }*+[o][F-]{\phantom{\alpha}}="a"
!{(0.5,1) }*+[o][F-]{\phantom{\alpha}}="b"
!{(-0.2,1) }*+{i_2}="c"
!{(-0.7,1) }*+{i_1}="d"
!{(1.2,1) }*+{i_3}="e"
!{(0,2) }*+{j_1}="g"
"b" : "a" 
"c" : "a"
"e" : "a"
"d" : "a"
"g" : "b"
"a" : @`{p+(0,-5),c+(3.2,10)} "b" 
}\qquad 
\xygraph{ 
!{<0cm,0cm>;<1cm,0cm>:<0cm,1cm>::} 
!{(0,0) }*+[o][F-]{\phantom{\alpha}}="a"
!{(0.5,1) }*+[o][F-]{\phantom{\alpha}}="b"
!{(0.3,2) }*+{i_2}="c"
!{(-0.1,2) }*+{i_1}="d"
!{(1.5,2) }*+{i_3}="e"
!{(-0.5,1) }*+{j_1}="g"
"b" : "a" 
"c" : "b"
"e" : "b"
"d" : "b"
"g" : "a"
"a" : @`{p+(0,-5),c+(3,10)} "b" 
}
\]
It is important to note, however, that once we decorate the vertex with three incoming edges by an operation $\alpha$ and the vertex with four incoming edges by an operation $\beta$, a Koszul sign arising of permuting $\alpha$ and $\beta$ will intervene, being in perfect agreement with the fact that all commutators are to be considered in the graded sense. 

Similarly to the PROP completion of an operad, to each wheeled operad $\calU$ one can associate a PROP $\calP(\calU)$ whose components are given by
 \[
(\calP(\calU))(m,n):=(\calU_\o^{\otimes n}\otimes S(\calU_\w))(m),
 \]
where the tensor and the symmetric powers are taken with respect to the Cauchy product, the vertical composition comes from the composition in $\calU_\o$ and the right $\calU_\o$-module structure on $\calU_\w$, and the horizontal composition is the product in the twisted associative algebra $\calU_\o^{\otimes \bullet}\otimes S(\calU_\w)$, where the algebra structure is uniquely determined by the structures of the factors and the condition that the two factors commute with each other. This PROP is also referred to as the PROP completion of the wheeled operad $\calU$; it automatically acquires a structure of a \emph{wheeled PROP}~\cite{MR2483835} coming from the composition maps and the trace maps in $\calU$.

There is also a notion of a wheeled cooperad, where all arrows are reversed (so that the structure of a wheeled operad on $\calW=\calW_\o\oplus\calW_\w$ consists of a cooperad structure on $\calW_\o$, a right $\calW_\o$-comodule structure on $\calW_\w$, and a right comodule morphism $\calW_\w\to\partial(\calW_\o)$ called cotrace, for which the composite $\calW_\w\to\partial(\calW_\o)\to \partial(\calW_\o)\otimes\partial(\calW_\o)$ is symmetric, that is has zero projection on ``co-commutators''). The coPROP completion of a wheeled cooperad $\calW$ is denoted $\calP^c(\calW)$.

\subsection{Wheeled bar construction and Koszulness}\label{sec:wheeledbar}

Given a wheeled operad $\calU$, its wheeled bar construction $\mathsf{B}^{\circlearrowleft}(\calU)$ is $\calT^{\circlearrowleft}(s\overline{\calU}_\o\oplus\overline{\calU}_\w)$ equipped with the differential given by the sum of all ways of collapsing edges of a graph and using the wheeled operad structure maps on the labels of the vertices that are affected by such collapses. 

Explicitly, $\calT^{\circlearrowleft}(s\overline{\calU}_\o\oplus\overline{\calU}_\w)$ has three parts:
\begin{itemize}
\item[-] the part corresponding to rooted trees, which is $\calT(s\overline{\calU}_\o)$,
\item[-] the part corresponding to cul-de-sac trees, which is $\overline{\calU}_\w\circ\calT(s\overline{\calU}_\o)$,
\item[-] the part corresponding to wheels, which can be viewed as 
 \[
 \Cyc(s\partial(\overline{\calU}_\o))\circ\calT(s\overline{\calU}_\o).
 \] 
\end{itemize}
(For the latter, one can remark that for a wheel graph, each vertex involved in its only cycle has a special incoming edge that is used in the cycle, so we may view the vertices involved in the cycle as labelled by $s\partial(\overline{\calU}_\o)$; the fact that we consider a wheel amounts to using the cyclic words $\Cyc(s\partial(\overline{\calU}_\o))$.) The differential of the bar construction can be described as follows:
\begin{itemize}
\item[-] on the rooted tree part $\calT(s\overline{\calU}_\o)$, it is the differential of the operadic bar construction $\mathsf{B}(\calU_\o)$,
\item[-] on the cul-de-sac tree part $\overline{\calU}_\w\circ\calT(s\overline{\calU}_\o)$, it is the differential of the operadic bar construction $\mathsf{B}(\overline{\calU}_\w,\calU_\o)$ with coefficients in the right module $\overline{\calU}_\w$,
\item[-] on the wheel part $\Cyc(s\partial(\overline{\calU}_\o))\circ\calT(s\overline{\calU}_\o)$, it is the sum of the differential of the cyclic complex $\Cyc(s\partial(\overline{\calU}_\o))$ of the twisted associative algebra $\partial(\overline{\calU}_\o)$, the differential of the operadic bar complex $\mathsf{B}(\Cyc(s\partial(\overline{\calU}_\o)),\calU_\o)$ with coefficients in the right module $\Cyc(s\partial(\overline{\calU}_\o))$, and the map to the cul-de-sac tree part that takes a wheel whose cycle is a loop and erases that loop, transforming the label of the corresponding vertex by the trace map. 
\end{itemize}
The signs in the formula for the differential come from the fact that the wheeled bar construction $\mathsf{B}^{\circlearrowleft}(\calU)$ is the cofree (conilpotent) wheeled cooperad cogenerated by $s\overline{\calU}_\o\oplus\overline{\calU}_\w$, and the requirement that the differential is the unique coderivation extending the edge contractions on trees with two vertices and the loop contraction on wheels with one vertex. (Derivations and coderivations for wheeled operads are defined by the usual compatibility with the structure operations; in particular, a derivation of a free object is uniquely determined by the restriction to generators, and a coderivation of a cofree conilpotent object is uniquely determined by the corestriction to cogenerators.)

A weight graded wheeled operad with all generators of weight one is said to be \emph{Koszul} if the homology of its wheeled bar construction is concentrated on the diagonal (weight equal to homological degree). To give some intuition for this notion, let us remark that, in the particular case where all generators of $\calU$ are operadic (that is, come from $\calU_\o$) and moreover are all given by binary operations, the Koszul property implies that all operadic relations are linear combinations of partial compositions of pairs of generators, and all trace relations are linear combinations of traces of generators.

\begin{remark}
As in the case of operads, the main reason to study the wheel bar construction and its homology comes from the fact that it is a basic invariant of the given wheeled operad in the homotopy category of differential graded wheeled operads, which one studies by combining the general philosophy of homotopical algebra \cite{MR0223432} with the usual methods one uses when working with algebras over coloured operads \cite{MR2342815,rezkthesis}, keeping in mind that it is convenient to think about groupoid-coloured operads \cite{MR3134040,MR4425832}. In short, similarly to what happens in the case of operads, the groupoid-coloured operad encoding wheeled operads is Koszul and Koszul self-dual (which one can easily establish using the methods of \cite{KWKoszul} or \cite{MR4425832}), and the bar construction gives one half of the bar-cobar adjunction between the categories of (reduced) augmented differential graded wheeled operads and conilpotent coaugmented differential graded wheeled cooperads; moreover, that adjunction is a Quillen equivalence, and, in particular, the wheeled cobar construction of the wheeled bar construction of a given wheeled operad $\calU$ is quasi-isomorphic to $\calU$, and moreover gives a functorial cofibrant replacement of $\calU$. Thus, the bar construction is the chain complex of indecomposables of a cofibrant replacement, and the homology of $\mathsf{B}^{\circlearrowleft}(\calU)$ is the Quillen homology of the wheeled operad $\calU$. 
\end{remark}

Let us discuss the bar construction in one very particular case which will be very informative for understanding the main results of this paper.

\begin{example}\label{ex:wheeledcomplalgbar}
Let $A$ be an augmented associative algebra with augmentation ideal $I$, and consider it as an operad that only has operations in arity one given by $A$. Then the zero completion of $A$ is just the algebra $A$ itself, and the bar construction $\mathsf{B}^{\circlearrowleft}(A)$ is $\mathsf{B}(A)\oplus\Cyc(sI)$. The differential of the bar construction is induced by the differential of the usual bar construction of $A$ on both terms, and $\Cyc(sI)$ computes (up to a degree shift by one) the cyclic homology of $I$, so its homology is $H_\bullet(\mathsf{B}(A))\oplus HC_{\bullet-1}(I)$. (Note that we use the convention of \cite{MR2954392} for which the bar construction of an augmented algebra $A$ is the cofree coalgebra on its augmentation ideal $I$.) 
\end{example}

\section{Stable homology and bar constructions of wheeled operads}\label{sec:stable-homology}

In this section, we shall prove the two theorems about stable homology of Lie algebras that we announced in the introduction. The proofs rely on the use of invariant theory of $\mathfrak{gl}(V)$ in the spirit of the classical Gelfand--Fuchs calculations \cite{MR0266195} which inspired a great number of calculations of homology of infinite-dimensional Lie algebras, see, for example, the papers \cite{MR0923136,MR0356082,MR0327856,MR2269758,MR3492045,MR1247289}. \\

The main tool for computing $\mathfrak{gl}(V)$-invariants is provided by the First Fundamental Theorem for $\mathfrak{gl}(V)$ that asserts that the subspace of $\mathfrak{gl}(V)$-invariants in $V^{\otimes r}\otimes (V^*)^{\otimes s}$ is nonzero only for $r=s$, and in that case, if we use the canonical identification of that space with its linear dual, is spanned by invariants of the form 
 \[
f_\sigma\colon v_1\otimes \cdots\otimes v_r\otimes \xi_1\otimes\xi_r\mapsto \xi_1(v_{\sigma(1)})\cdots \xi_r(v_{\sigma(r)}),\quad \sigma\in S_r.
 \]
Moreover, the Second Fundamental Theorem for $\mathfrak{gl}(V)$ asserts that those invariants are linearly independent as long as $\dim(V)\ge r$. Proofs of both of these theorems are contained in \cite{MR0000255}, and are at the core of the Schur--Weyl duality between representations of symmetric groups and general linear groups. For an exposition of these results that is very close in spirit to the present work, the reader is also invited to consult \cite{MR2501578}. 

\subsection{Stable homology of the Lie algebra \texorpdfstring{$\Der^+(\calO(V))$}{Der+O} with bivariant}

In this section, we shall relate the stable homology of Lie algebras $\Der^+(\calO(V))$ with bivariant coefficients to the homolog of the wheeled bar construction of $\calO$. It is one of the central results on this paper that allowed us to create a new powerful connection between two \emph{a priori} unrelated topics.  

The following result generalizes work of Fuchs \cite{MR0725416} and directly relates the Chevalley--Eilenberg complexes to the wheeled bar construction.

\begin{theorem}\label{th:graphcx1}
The $\mathfrak{gl}(V)$-invariant Chevalley--Eilenberg complexes 
 \[
C_\bullet(\Der^+(\calO(V)),\Hom(V^{\otimes \bullet_1},V^{\otimes \bullet_2}))^{\mathfrak{gl}(V)}
 \]
are stably isomorphic to the coPROP completion of the cobar construction $\mathsf{B}^{\circlearrowleft}(\calO)$. This isomorphism is natural with respect to operad morphisms.
\end{theorem}

\begin{proof}
Note that
\[
\Der(\calO(V))\cong \bigoplus_{k\ge 1}\Hom(V,\calO(k)\otimes_{\k S_k}V^{\otimes k})\cong \bigoplus_{k\ge 1}V^*\otimes\calO(k)\otimes_{\k S_k}V^{\otimes k},
 \]
so that we have  
\begin{multline*}
C_\bullet(\Der^+(\calO(V)), \Hom(V^{\otimes p},V^{\otimes q}))\cong \\
S(s\Der^+(\calO(V)))\otimes\Hom(V^{\otimes p},V^{\otimes q})\cong\\  
\bigoplus_{\substack{k\ge 1,\\ n_1,\ldots,n_k\ge 1\\ p_1,\ldots,p_k\ge 0}}\bigotimes_{i=1}^kS^{p_i}(sV^*\otimes\overline{\calO}(n_i)\otimes_{\k S_{n_i}}V^{\otimes n_i})\otimes (V^*)^{\otimes p}\otimes V^{\otimes q}.   
\end{multline*}
Moreover, the summand 
 \[
\bigotimes_{i=1}^kS^{p_i}(sV^*\otimes\overline{\calO}(n_i)\otimes_{\k S_{n_i}}V^{\otimes n_i})\otimes (V^*)^{\otimes p}\otimes V^{\otimes q} 
 \]
is, as a $\mathfrak{gl}(V)$-module, a submodule of several copies of $V^{\otimes N}\otimes (V^*)^{\otimes M}$ where
 \[
N=\sum_{i=1}^kn_ip_i+q, \quad M=\sum_{i=1}^k p_i +p.
 \] 
We also note that our chain complex is bi-graded:
 \[
C^{(w)}_d(\Der^+(\calO(V)),\Hom(V^{\otimes p},V^{\otimes q}))^{\mathfrak{gl}(V)} 
 \]
consists of elements of weight $w$ (computed out of the weight on the Lie algebra $\Der^+(\calO(V))$) and of homological degree $d$. In the notation above, we have $w=\sum_{i=1}^k (n_i-1)p_i$ and $d=\sum_{i=1}^k p_i$, so in particular $w+d=\sum_{i=1}^k n_ip_i$. Thus, our above formulas can be written as 
 \[
N=w+d+q, M=d+p. 
 \] 
Instead of thinking of sums of tensor products of symmetric powers above, we shall visualize our chain complex as the vector space spanned by all linear combinations of sets of ``LEGO pieces'' of the following three types: 
\begin{enumerate}
\item a ``corolla'', that is a vertex $v$ with one incoming half-edge and $n_v\ge 2$ outgoing half-edges, where $v$ carries a label from $\overline{\calO}(\text{out}(v))$, the incoming half-edge is labelled by an element $\xi\in V^*$, and each outgoing half-edge $e$ is labelled by an element $u_e\in V$ (here we implicitly impose an equivalence relation saying that if we simultaneously act on the label of the vertex and on labels of outgoing half-edges by the same permutation, the corolla does not change);   
\item a ``source'', that is a vertex $v$ with one outgoing half-edge and no incoming half-edges, where $v$ uniquely corresponds to an element of $\{1,\ldots,q\}$, and the outgoing half-edge is labelled by an element $u\in V$;   
\item a ``sink'', that is a vertex $v$ with one incoming half-edge and no outgoing half-edges, where $v$ uniquely corresponds to an element of $\{1,\ldots,p\}$, and the incoming half-edge is labelled by an element $\xi\in V^*$.
\end{enumerate}
Corollas with $n_v$ outgoing half-edges are used to represent elements of the vector space
$sV^*\otimes\overline{\calO}(n_v)\otimes_{\k S_{n_v}}V^{\otimes n_v}$; they have homological degree $1$, and so reordering them creates a sign. The sources and the sinks appear in the natural order of their vertex labels, reproducing the vector space $(V^*)^{\otimes p}\otimes V^{\otimes q}$.

From the First and the Second Fundamental Theorems for $\mathfrak{gl}(V)$, it follows that to have non-zero invariants at all we must have $N=M$, and that for the ``stable range'' $\dim(V)\ge M=d+p$ the $\mathfrak{gl}(V)$-invariants in this module are isomorphic to a vector space with a basis of (equivalence classes of) graphs obtained by assembling the LEGO pieces together, that is directed decorated graphs $\Gamma$ obtained by matching all the incoming half-edges with all the outgoing half-edges. Specifically, we obtain that each vertex $v$ of $\Gamma$ is one of the following three types:
\begin{enumerate}
\item a corolla $v$ with one incoming half-edge and $n_v\ge 1$ outgoing half-edges labelled by an element from $\overline{\calO}(\text{out}(v))$;
\item a source with one outgoing half-edge uniquely corresponding to an element of $\{1,\ldots,q\}$;    
\item a sink with one incoming half-edge uniquely corresponding to an element of $\{1,\ldots,p\}$.  
\end{enumerate}
Two such graphs are equivalent if there is an isomorphism of them that agrees with all labels. As above, corollas have homological degree $1$, and so reordering them creates a sign. Such a graph corresponds to an invariant that pairs the copies of $V$ and $V^*$ according to the edges between the vertices in the graph. 

Similarly to Lemma \ref{lm:fuchs}, we see that connected graphs that appear as basis elements in the stable range are precisely the connected graphs appearing as basis elements of the wheeled bar construction $\mathsf{B}^{\circlearrowleft}(\calO)$. Indeed, since we consider $\calO$ as a wheeled operad with the zero trace map, graphs containing a cul-de-sac vertex make zero contribution to the result of applying the monad of rooted trees and wheels. Moreover, the labels of corollas belong to $\overline{\calO}$, and reordering corollas creates a sign, which is in agreement with the unordered tensor products in the definition of $\mathsf{B}^{\circlearrowleft}(\calO)=\mathsf{T}^{\circlearrowleft}(s\overline{\calO})$.

Finally, the differential 
\begin{multline*}
C^{(w)}_{d+1}(\Der^+(\calO(V)),\Hom(V^{\otimes p},V^{\otimes q}))^{\mathfrak{gl}(V)}\to \\ 
C^{(w)}_d(\Der^+(\calO(V)),\Hom(V^{\otimes p},V^{\otimes q}))^{\mathfrak{gl}(V)} 
\end{multline*}
admits an elegant formula in the stable range $\dim(V)>d+p$, where the strict inequality comes from the fact that we need to know basis elements in homological degree $d+1$ as well. First of all, note that since the action of $\Der(\calO(V))$ on $\mathfrak{gl}(V)$-modules is via the augmentation, the subalgebra $\Der^+(\calO(V))$ acts on all such modules trivially, and so the Chevalley--Eilenberg differential only uses the Lie bracket in $\Der^+(\calO(V))$. Next, we recall that the Lie algebra of derivations comes from a pre-Lie structure, and, if we think of derivations in terms of linear combinations of corollas representing them, the pre-Lie product $D\triangleleft D'$ is equal to the sum of elements corresponding to pairing one of the outputs of $D$ with the input of $D'$  and computing the corresponding contraction $V^*\otimes V\to \k$ simultaneously with the partial composition in $\calO$. For the complex of invariants, this means that the computation of each Lie bracket $[D,D']$ appearing in the standard formula for the Chevalley--Eilenberg differential will be made of two separate parts: the part $D\triangleleft D'$ via the edge contraction in the graphs where the input of the corolla corresponding to $D'$ is matched with one of the outputs of the corolla corresponding to $D$, and $D'\triangleleft D$ otherwise. Note that if an edge is a loop going from a vertex to itself, contracting that edge does not correspond to computing the Lie bracket and hence does not contribute to the differential. We see that the differential matches precisely the differential of $\mathsf{B}^{\circlearrowleft}(\calO)$: there the differential is also made of edge contractions, and contracting loops gives zero precisely because we consider $\calO$ as a wheeled operad with the zero trace map. (There are several standard ways to check that the signs match; the reader familiar with graph complexes already knows them, and the psychologically easiest way for other readers would be to dualize and consider the Chevalley--Eilenberg cohomological complex and the cobar construction; for each of them the differential is a derivation and hence is ``local'', that is applies to each corolla individually, and the fact that the two differentials agree on corollas is almost obvious.) 

Passing to disconnected graphs corresponds to computing tensor powers (for trees) and symmetric powers (for wheels) with respect to the Cauchy product, which is precisely the coPROP completion of the wheeled bar construction viewed as a wheeled cooperad. The statement on naturality is obvious.
\end{proof}

The theorem we just proved implies the following statement, which is the main result of this section.

\begin{theorem}\label{th:main1}
Let $\calO$ be an augmented operad, and $V$ a finite-dimensional vector space. The vector spaces 
 \[
H_\bullet(\Der^+(\calO(V)),\Hom(V^{\otimes \bullet_1},V^{\otimes \bullet_2}))^{\mathfrak{gl}(V)}
 \]
stabilize as $\dim(V)\to\infty$, with the stable value given by the coPROP completion of the wheeled cooperad $H_\bullet(\mathsf{B}^{\circlearrowleft}(\calO))$. This isomorphism is natural with respect to operad morphisms.
\end{theorem}

\begin{proof}
As we saw in the proof of Theorem \ref{th:graphcx1}, if we fix $w,d,p,q$, then as long as $\dim(V)>d+p$, the components of weight $w$ and homological degrees $d$ and $d+1$ of the complex $C_\bullet(\Der^+(\calO(V)),\Hom(V^{\otimes p},V^{\otimes q}))^{\mathfrak{gl}(V)}$ are nonzero only for $p=w+q$ and are isomorphic to the corresponding components of the wheeled bar construction $\mathsf{B}^{\circlearrowleft}(\calO)$; moreover, the differential of the Chevalley--Eilenberg complex is identified under that isomorphism with the differential of the wheeled bar construction. Since computing homology commutes with taking $\mathfrak{gl}(V)$-invariants, we conclude that for such $w,d,p,q$, we have 
 \[
H^{(w)}_d(\Der^+(\calO(V)),\Hom(V^{\otimes p},V^{\otimes q}))^{\mathfrak{gl}(V)}\cong H^{(w)}_d(\mathsf{B}^{\circlearrowleft}(\calO)).
 \] 
It remains to apply the K\"unneth formula to obtain the desired result for the coPROP completion.
\end{proof}

Let us show how the classical Loday--Quillen--Tsygan theorem arises in this context.

\begin{corollary}[Loday--Quillen \cite{MR0780077}, Tsygan \cite{MR0695483}]\label{cor:LQT}
Let $A$ be a unital $\k$-algebra. We have a Hopf algebra isomorphism
 \[
H_\bullet(\mathfrak{gl}(A),\k)\cong S^c(HC_{\bullet-1}(A)).     
 \]
\end{corollary}

\begin{proof}
Let us consider the algebra $A_+$ which is obtained from $A$ by adjoining a unit (even though $A$ is already unital). This algebra has the canonical augmentation for which the augmentation ideal $\overline{A_+}$ is $A$. If we consider $A_+$ as an operad $\calO$ concentrated in arity one,  Example \ref{ex:wheeledcomplalgbar}, 
 \[
H_\bullet(\mathsf{B}^{\circlearrowleft}(\calO)_\o)=0, \quad H_\bullet(\mathsf{B}^{\circlearrowleft}(\calO)_\w)\cong
 HC_{\bullet-1}(\overline{A_+})\cong
 HC_{\bullet-1}(A),
 \]
since the augmentation ideal $\overline{\calO}=\overline{A_+}$ is our unital algebra $A$ (in particular, the operadic part of the bar construction is acyclic).
Moreover, $HC_{\bullet-1}(A)$ is the $(0,0)$-component of $H_\bullet\mathsf{B}^{\circlearrowleft}(\calO)_\w)$, and hence the coPROP completion of the homology of the bar construction is precisely $S^c(HC_{\bullet-1}(A))$ supported at the $(0,0)$-component. Theorem \ref{th:main1} now implies that the stable limit of
 \[
H_\bullet(\Der^+(\calO(V)),\k)^{\mathfrak{gl}(V)}
 \]
is isomorphic to $S^c(HC_{\bullet-1}(A))$. According to Example \ref{ex:divforalg}, in the case $V=\k^n$, we have $\Der^+(\calO(V))\cong\mathfrak{gl}_n(A)$, so  
 \[
\lim_{n\to\infty}(H_\bullet(\mathfrak{gl}_n(A),\k))^{\mathfrak{gl}_n(\k)}\cong S^c(HC_{\bullet-1}(A)).
 \]
Since the algebra $A$ is unital, the Lie algebra $\mathfrak{gl}_n(A)$ contains $\mathfrak{gl}_n(\k)$ as a Lie subalgebra. It is known \cite[Sec.~1.6.C]{MR0874337} that the action of any Lie algebra in its homology is trivial, so we have
 \[
\lim_{n\to\infty}H_\bullet(\mathfrak{gl}_n(A),\k)\cong\lim_{n\to\infty}(H_\bullet(\mathfrak{gl}_n(A),\k))^{\mathfrak{gl}_n(\k)}\cong S^c(HC_{\bullet-1}(A)).
 \]

This gives the coalgebra structure on $H_\bullet(\mathfrak{gl}(A),\k)$, and the Hopf algebra isomorphism is easy to obtain by recalling that the product on $H_\bullet(\mathfrak{gl}(A))$ comes from the canonical embeddings $\mathfrak{gl}_n(\k)\oplus\mathfrak{gl}_m(\k)\hookrightarrow\mathfrak{gl}_{n+m}(\k)$.  
\end{proof}

\begin{remark}
The last sentence of the above proof suggests that the coPROP structure on the stable homology computed in Theorem \ref{th:main1}  arises naturally if we work in the category $\calS(\k)$ of \cite{Powell21}; this observation will be explored elsewhere. 
\end{remark}

Let us also record a result relating the computation of Theorem \ref{th:main1} to calculation of the homology of the Lie algebra of all derivations with bivariant coefficients.

\begin{proposition}\label{prop:thmain1}
Let $\calO$ be an augmented operad, and $V$ a finite-dimensional vector space. The homology 
 \[
H_\bullet(\Der(\calO(V)),\Hom(V^{\otimes p},V^{\otimes q})) 
 \] 
is isomorphic to the cofree $H_\bullet(\mathfrak{gl}(\k))$-comodule cogenerated by 
 \[
H_\bullet(\Der^+(\calO(V)),\Hom(V^{\otimes p},V^{\otimes q}))^{\mathfrak{gl}(V)}. 
 \] 
The latter is isomorphic to  
 \[
(H_\bullet(\Der^+(\calO(V)),\k)\otimes\Hom(V^{\otimes p},V^{\otimes q}))^{\mathfrak{gl}(V)}.
 \] 
\end{proposition}

In the particular case $p=q=0$, our result concerns the homology of the Lie algebra $\Der(\calO(V))$ with trivial coefficients. Since the action of the subalgebra $\mathfrak{gl}(V)\subset\Der(\calO(V))$ on the homology is trivial, this homology is easily seen to be covered by the Loday--Quillen--Tsygan theorem in the case $A=\calO(1)$, as already remarked in~\cite{Mah03}. For other values of $p,q$, the situation becomes more complex, and, as a consequence, the Lie algebra $\Der^+(\calO(V))$ has many ``interesting'' homology classes even in the case of trivial coefficients, as we shall see below. An interesting natural question that will be addressed elsewhere is to compute the homology of the Lie algebra $\Der(\calO(V))$ for various non-augmented operads $\calO$; note that for $\calO=\uCom$, that homology can be identified with homology of a natural topological space, see \cite{MR0266195} for details. 

\begin{proof}
The argument is very similar to its version in the case $\calO=\Com$, where it essentially goes back to the work of Losik \cite{MR0312518} (see also \cite[Th.~2.2.8]{MR0874337}). Recall \cite[Th.~1.5.1]{MR0874337} that the Hochschild--Serre spectral sequence for a Lie algebra $\mathfrak{g}$, its Lie subalgebra $\mathfrak{h}$, and a $\mathfrak{g}$-module $M$ arises from the filtration $F_\bullet$ of the Chevalley--Eilenberg complex $C_\bullet(\mathfrak{g},M)=S(s\mathfrak{g})\otimes M$ for which 
 \[
F_mC_{m+n}(\mathfrak{g},M)= S^m(s\mathfrak{h})\otimes S^n(s\mathfrak{g})\otimes M,
 \]
or, in plain words, we consider linear combinations of chains involving at least $m$ factors from the subalgebra $\mathfrak{h}$. In general, only the first page
 \[
E^1_{m,n}=H_n(\mathfrak{h},S^m(s\mathfrak{g}/\mathfrak{h})\otimes M) 
 \]
is tractable, but in our case, there are several special features that make the computation possible. First of all, the Lie subalgebra $\mathfrak{h}=\Der^+(\calO(V))$ is an ideal of $\mathfrak{g}=\Der(\calO(V))$, and so it acts trivially on
 \[
\Der(\calO(V))/\Der^+(\calO(V))\cong\mathfrak{gl}(V).
 \] 
Moreover, $\Der^+(\calO(V))$ is the kernel of the augmentation map to $\mathfrak{gl}(V)$, and hence acts trivially on the module $M=\Hom(V^{\otimes p},V^{\otimes q})$. Thus we immediately obtain
 \[
E^2_{m,n}=H_m(\mathfrak{gl}(V),H_n(\Der^+(\calO(V)),\k)\otimes\Hom(V^{\otimes p},V^{\otimes q})).
 \]
Note that for any $\mathfrak{gl}(V)$-module $R$ that is a submodule of a direct sum of several copies of $V^{\otimes N}\otimes (V^*)^{\otimes M}$, \cite[Th.~2.1.2]{MR0874337} ensures that we have, identifying the $\mathfrak{gl}(V)$-invariants with the corresponding coinvariants,     
 \[
H_\bullet(\mathfrak{gl}(V),R)\cong H_\bullet(\mathfrak{gl}(V),\k)\otimes R^{\mathfrak{gl}(V)},
 \]
which shows that  
 \[
E^2_{m,n}\cong H_m(\mathfrak{gl}(V),\k)\otimes \left(H_n(\Der^+(\calO(V)),\k)\otimes\Hom(V^{\otimes p},V^{\otimes q})\right)^{\mathfrak{gl}(V)}.
 \]
Let us show that our spectral sequence abuts on the second page. Note that the $C_\bullet(\Der(\calO(V)),\k)$-comodule structure on the complex 
 \[
C_\bullet(\Der(\calO(V)),\Hom(V^{\otimes p},V^{\otimes q})),
 \]
combined with the augmentation $\Der(\calO(V))\twoheadrightarrow \mathfrak{gl}(V)$, makes that complex a $C_\bullet(\mathfrak{gl}(V),\k)$-comodule. We can further project to invariants 
 \[
C_\bullet(\mathfrak{gl}(V),\k)^{\mathfrak{gl}(V)}\cong H_\bullet(\mathfrak{gl}(V),\k)
 \] 
and consider the Chevalley--Eilenberg complex $C_\bullet(\Der(\calO(V)),\Hom(V^{\otimes p},V^{\otimes q}))$ as a $H_\bullet(\mathfrak{gl}(V),\k)$-comodule. 
Moreover, since the Hochschild--Serre filtration is defined by bounding from below the number of factors from $\Der^+(\calO(V))$, the $H_\bullet(\mathfrak{gl}(V),\k)$-comodule structure is compatible with that filtration, making our spectral sequence a sequence of $H_\bullet(\mathfrak{gl}(V),\k)$-comodules and their maps. Moreover, our analysis of the $E^2$ term shows that it is cogenerated by 
 \[
E^2_{0,j}\cong \left(H_j(\Der^+(\calO(V)),\k)\otimes \Hom(V^{\otimes p},V^{\otimes q})\right)^{\mathfrak{gl}(V)},
 \] 
so the proof will be complete if we show that all these elements survive on all pages of our spectral sequence. This is equivalent to proving that the natural map
\begin{multline*}
\left(H_\bullet(\Der^+(\calO(V)),\k)\otimes \Hom(V^{\otimes p},V^{\otimes q})\right)^{\mathfrak{gl}(V)}\to\\ 
H_\bullet(\Der(\calO(V)),\Hom(V^{\otimes p},V^{\otimes q})) 
\end{multline*} 
is injective. To show that, we note that the map   
\begin{multline*} 
C_\bullet(\Der(\calO(V)),\Hom(V^{\otimes p},V^{\otimes q})) \to\\
\left(C_\bullet(\Der^+(\calO(V)),\k)\otimes \Hom(V^{\otimes p},V^{\otimes q})\right)^{\mathfrak{gl}(V)}
\end{multline*} 
that sends wedge products involving at least one element of $\mathfrak{gl}(V)$ to zero and projects to invariants is manifestly its left inverse on the level of homology. 
\end{proof}

\subsection{Wheeled bar homology of operads and cyclic homology}

To convince the reader that Theorem \ref{th:main1} is not a mere reformulation of a complicated problem, let us establish a new somewhat general result about the wheeled bar homology of an operad $\calO$. For that, let us denote
 \[
\partial(\calO)_0:=\partial(\calO)\circ_{\calO}\mathbbold{1},
 \]
where $\mathbbold{1}$ is equipped with the left $\calO$-module structure given by the augmentation. The meaning of this construction from the point of view of algebras over operads is as follows: $\partial(\calO)_0$ is the Schur functor of universal multiplicative enveloping algebras of trivial $\calO$-algebras (algebras on which the action of $\calO$ factors through the augmentation); in particular, it is a twisted associative algebra.

\begin{theorem} \label{th:calchom}
Let $\calO$ be an augmented operad for which $\partial(\calO)$ is free as a right $\calO$-module. We have 
 \[
H_\bullet(\mathsf{B}^{\circlearrowleft}(\calO)_\o)\cong H_\bullet(\mathsf{B}(\calO)),\quad H_\bullet(\mathsf{B}^{\circlearrowleft}(\calO)_\w)\cong HC_{\bullet-1}(\partial(\overline{\calO})_0).
 \]
\end{theorem}

\begin{proof}
Since we have $\mathsf{B}^{\circlearrowleft}(\calO)_\o=\mathsf{B}(\calO)$, the first claim is tautological. Let us consider $\mathsf{B}^{\circlearrowleft}(\calO)_\w$. As explained in Section \ref{sec:wheeledbar}, the underlying species of this chain complex is 
 \[
\Cyc(s\partial(\overline{\calO}))\circ \calT(s\calO)
 \]
with the differential that is the sum of the differential of $\Cyc(s\partial(\overline{\calO}))$ and the differential of $\mathsf{B}(\Cyc(s\partial(\overline{\calO})),\calO)$; the last part of the differential is zero since all trace maps are zero in our case. These two differentials anti-commute (this reflects the fact that the structures of a twisted associative algebra and a right $\calO$-module on $\partial(\calO)$ commute). To prove the statement of the theorem, let us consider the increasing filtration of our complex by the number of edges involved in the wheel. The first differential decreases this number, and the second differential preserves it, so in the associated graded complex, only the second differential survives, making our complex isomorphic to 
 \[
\Cyc(s\mathsf{B}(\partial(\overline{\calO}),\calO)),
 \]
where the differential is that of the operadic bar construction with coefficients in the right module $\partial(\overline{\calO})$.

Note that the species $\partial(\calO)_0$ is isomorphic to the minimal species of generators of $\partial(\calO)$ as a right $\calO$-module, so our assumption implies that we have a right module isomorphism 
 \[
\partial(\calO)\cong \partial(\calO)_0\circ \calO.
 \] 
Moreover, $\partial(\mathbbold{1})$ has no slots for the right module action, so we have a splitting of right $\calO$-modules
 \[
\partial(\calO)\cong\partial(\mathbbold{1})\oplus\partial(\overline{\calO})\cong\partial(\mathbbold{1})\oplus\partial(\overline{\calO}_0)\circ \calO.
 \]
In particular, the homological criterion of freeness \cite{MR4300233} shows that the homology of $\mathsf{B}(\partial(\overline{\calO}),\calO)$ is concentrated in degree zero and identifies with $\partial(\overline{\calO})_0$. As a consequence, the spectral sequence abuts on the following page, where one has to compute the homology of the differential induced by the first differential above, that is the differential of the cyclic complex of the twisted associative algebra $\partial(\overline{\calO})_0$, completing the proof.
\end{proof}

The key aspect on which Theorem \ref{th:calchom} relies is the freeness property of $\partial(\calO)$ as a right $\calO$-module; this property previously appeared in work of Khoroshkin~\cite{MR4381941} who used the result of the author of the present paper and Tamaroff \cite{MR4300233} to show that this condition is necessary and sufficient for the functorial PBW property of multiplicative universal enveloping algebras of $\calO$-algebras. Note that in \cite{MR4381941} it is shown that the operad of Poisson algebras does not have the PBW property, and the same operad exhibits nontrivial cycles in the wheeled completion of its bar construction, which is an unpublished result obtained independently by Bruinsma (in a bachelor thesis at the University of Amsterdam in 2010, see \cite{bruinsma2023cohomology}) and Khoroshkin (private communication). This surprising coincidence of counterexamples is precisely what made us discover Theorem \ref{th:calchom}.\\

Let us give several examples of how  Theorem \ref{th:calchom} can be applied. 

\begin{example}
Suppose that the operad $\calO$ only has unary operations, meaning that it is simply an augmented associative algebra. In this case, we have no slots for the right $\calO$-action on $\partial(\calO)$, so that action factors through the augmentation, and $\partial(\calO)_0\cong\partial(\calO)$, so we recover the result of Example \ref{ex:wheeledcomplalgbar}.
\end{example}

We proceed with a result which is essentially the wheeled Koszul property of $\Com$ considered as a Koszul operad, first proved in \cite{MR2483835}.

\begin{corollary}\label{cor:barcom}
We have 
 \[
H_\bullet(\mathsf{B}^{\circlearrowleft}(\Com)_\o)\cong\Com^{\ac}, \quad H_\bullet(\mathsf{B}^{\circlearrowleft}(\Com)_\w)\cong\Cyc(s\mathbbold{1}). 
 \]
\end{corollary}

\begin{proof}
Indeed, we have the twisted associative algebra isomorphism $\partial(\Com)\cong\uCom$, so $\partial(\overline{\Com})\cong\Com$ and $\partial(\overline{\Com})_0\cong\mathbbold{1}$ with zero multiplication, so its cyclic homology as a twisted associative algebra is the cyclic complex $s^{-1}\Cyc(s\mathbbold{1})$. 
\end{proof}

The following result is essentially a new proof of \cite[Th.~6.1.1]{MR2483835}, which shows that the operad $\Ass$ is not Koszul as a wheeled operad, but the homology of its bar construction is ``almost'' concentrated on the diagonal.
\begin{corollary}\label{cor:barass}
We have 
\begin{gather*}
H_\bullet(\mathsf{B}^{\circlearrowleft}(\Ass)_\o)\cong\Ass^{\ac},\\
H_\bullet(\mathsf{B}^{\circlearrowleft}(\Ass)_\w)\cong\Cyc(s\mathbbold{1})\oplus \Cyc(s\mathbbold{1})\oplus(\k s^{-1}\oplus \k)\otimes\Cyc(s\mathbbold{1})\otimes\Cyc(s\mathbbold{1}). 
\end{gather*}
\end{corollary}

\begin{proof}
Indeed, we know that $\partial(\Ass)\cong\uAss\otimes\uAss$, and from that it is easy to see that we have 
 \[
\partial(\Ass)_0=(\partial(\Com)_0)^{\otimes2}.
 \] 
Now, we have, in the standard cyclic homology notations, 
 \[
HC_{\bullet}(\partial(\Com)_0)\cong\k[u]\oplus s^{-1}\Cyc(s\mathbbold{1}),  
 \]
so by an analogue of \cite[Prop.~4.4.8]{MR1600246} for twisted associative algebras, we have 
\begin{multline*}
HC_{\bullet}(\partial(\Com)_0^{\otimes2})\cong HC_{\bullet}(\partial(\Com)_0)\oplus HH_{\bullet}(\partial(\Com)_0)\otimes s^{-1}\Cyc(s\mathbbold{1})
\cong\\ \left(HC_{\bullet}(\k)\oplus s^{-1}\Cyc(s\mathbbold{1})\right)\oplus \left(\k\oplus(\k\oplus \k s)\otimes s^{-1}\Cyc(s\mathbbold{1})\right)\otimes s^{-1}\Cyc(s\mathbbold{1}),
\end{multline*}
where the last equality is immediate if one uses the reduced Hochschild complex to compute the homology, so
 \[
HC_{\bullet-1}(\partial(\overline{\Ass})_0)\cong \Cyc(s\mathbbold{1})\oplus \Cyc(s\mathbbold{1})\oplus (\k s^{-1}\oplus \k)\otimes(\Cyc(s\mathbbold{1})\otimes \Cyc(s\mathbbold{1})).  
 \]
\end{proof}

The following result is, to the best of our knowledge, new. It is the first natural step towards computing the minimal model of $\Lie$ as a wheeled operad; we shall address this question elsewhere.

\begin{corollary}\label{cor:barlie}
We have 
 \[
H_\bullet(\mathsf{B}^{\circlearrowleft}(\Lie)_\o)\cong\Lie^{\ac},\quad H_\bullet(\mathsf{B}^{\circlearrowleft}(\Lie)_\w)\cong\mathrm{Hook}, 
 \]
where $\mathrm{Hook}$ is the linear species with the homological degree $d$ part of $\mathrm{Hook}(n)$ being the $S_n$-module corresponding to the Young diagram $(n-d+1,1^{d-1})$, $d\ge 1$. 
\end{corollary}

\begin{proof}
As indicated in Example \ref{ex:ULie}, we have an isomorphism of twisted associative algebras $\partial(\Lie)\cong\uAss$; moreover, there is a PBW isomorphism of right $\Lie$-modules $\partial(\Lie)\cong\uCom\circ\Lie$. Consequently, we have an isomorphism of twisted associative algebras $\partial(\Lie)_0\cong\uCom$, and computing the cyclic homology is essentially computing the cyclic homology of polynomial algebras. The answer (for the reduced homology) is given by $\Com(\mathbbold{1}\oplus s\mathbbold{1})/\mathrm{Im}(d_{dR})$, where $d_{dR}$ is the unique derivation extending $s\colon \mathbbold{1}\to s\mathbbold{1}$. Note that we can explicitly compute the Cauchy product 
 \[
S^p(\mathbbold{1})\otimes S^{q}(s\mathbbold{1})\cong S^p\otimes S^{1^q}\cong S^{p,1^q}\oplus S^{p+1,1^{q-1}}
 \]
using the Pieri rule for Schur functions \cite{MR1354144}, and it is almost obvious that $d_{dR}$ sends $S^{p,1^q}\subset S^p(\mathbbold{1})\otimes S^{q}(s\mathbbold{1})$ to the same submodule of $S^{p-1}(\mathbbold{1})\otimes S^{q+1}(s\mathbbold{1})$. Thus, as a linear species, we have
 \[
HC_d^{(w)}(\partial(\overline{\Lie})_0)\cong S^{w-d,1^d}.
 \] 
Since we need $HC_{\bullet-1}$, shifting $d$ by one completes the proof. 
\end{proof}

Let us also record a new result of the same flavour concerning the operad of pre-Lie algebras.

\begin{corollary}\label{cor:barpl}
We have
\begin{gather*}
H_\bullet(\mathsf{B}^{\circlearrowleft}(\PL)_\o)\cong\PL^{\ac},\\ 
H_\bullet(\mathsf{B}^{\circlearrowleft}(\PL)_\w)\cong\mathrm{Hook}\oplus \uCom(\mathbbold{1}\oplus s\mathbbold{1})\otimes \Cyc(\mathbbold{1}). 
\end{gather*}
 \end{corollary}

\begin{proof}
It is known \cite{MR4381941,MR2078718} that $\partial(\PL)$ is a free right $\PL$-module with the species of generators $\uCom\otimes\uAss$, and 
 \[
HC_\bullet(\uAss)=\k[u]\otimes\k\oplus\Cyc(\mathbbold{1}),
 \]
so by an analogue of \cite[Prop.~4.4.8]{MR1600246}, we have 
\begin{multline*}
HC_{\bullet}(\uCom\otimes\uAss)\cong HC_{\bullet}(\uCom)\otimes\k\oplus HH_{\bullet}(\uCom)\otimes \Cyc(\mathbbold{1})
\cong\\ HC_{\bullet}(\uCom)\oplus \uCom(\mathbbold{1}\oplus s\mathbbold{1})\otimes \Cyc(\mathbbold{1}),
\end{multline*} 
so 
 \[
HC_{\bullet-1}(\partial(\overline{\PL})_0)\cong \mathrm{Hook}\oplus \uCom(\mathbbold{1}\oplus s\mathbbold{1})\otimes \Cyc(\mathbbold{1}),
 \]
as required. 
\end{proof}

\subsection{Applications to the homology of \texorpdfstring{$\Der^+(\calO(V))$}{Der+O} with constant coefficients} In this section, we shall make use of Theorem \ref{th:main1} to obtain some concrete information about the homology of the Lie algebra $\Der^+(\calO(V))$ with constant coefficients. That is particularly important for estimates of homology in degrees $1$ and $2$, for the following reason. If the operad~$\calO$ has an extra weight grading for which the augmentation ideal consists of elements of strictly positive weight (this happens, for example, if $\calO$ has no unary operations other than multiples of the identity), one can use that grading on the algebra $\Der^+(\calO(V))$ to give the homology an additional weight grading, and use that grading to identify the first homology with the minimal space of generators of that Lie algebra and the second homology with the minimal space of relations between those generators \cite[Prop.~A.9]{MR0874337}.

The main idea that we use is that, for a representation $M$ of $\mathfrak{gl}(V)$ that is contained in a finite sum of tensor products of copies of $V$ and $V^*$, one can completely determine the structure of $M$ from the spaces of ``matrix elements'' 
 \[
(M\otimes (V^*)^{\otimes p}\otimes V^{\otimes q})^{\mathfrak{gl}(V)}
 \]
for various $p,q$. In particular, as we saw above, each weight graded component $H_k^{(w)}(\Der^+(\calO(V)),\k)$ has a natural $\mathfrak{gl}(V)$-module structure satisfying this condition. Thus, the $\mathfrak{gl}(V)$-module structure is fully captured by the space of the matrix elements
 \[ 
(H_\bullet(\Der^+(\calO(V)),\k)\otimes\Hom(V^{\otimes p},V^{\otimes q}))^{\mathfrak{gl}(V)}
 \] 
for various values of $p$ and $q$. To give an illustration of how this works, we shall now show how to use Theorem \ref{th:main1} to recover the Fuchs stability theorem for the Lie algebra $L_1(n)$ of polynomial vector fields that vanish twice at the origin. (We have already seen that Theorem \ref{th:main1} easily implies the Loday--Quillen--Tsygan theorem.)

\begin{corollary}[Fuchs \cite{MR0725416}]\label{cor:newFuchs}
Suppose that $n>r+2d$. Then 
 \[
H^{(r)}_d(L_1(n),\k)=0 \text{ for } r\ne d.
 \]
\end{corollary}

\begin{proof}
First, let us note that we manifestly have $L_1(n)=\Der^+(\Com(V))$ for $V=\k^n$. Taking matrix elements of representations immediately implies that the vector space $H^{(r)}_d(\Der^+(\Com(V)),\k)$ vanishes if and only if
 \[
\left(H^{(r)}_d(\Der^+(\Com(V)),\k)\otimes\Hom(V^{\otimes p},V^{\otimes q}))\right)^{\mathfrak{gl}(V)}=0 \text{ for all } p,q. 
 \] 
This vector space is, according to Theorem \ref{th:main1}, stably isomorphic to the $(p,q)$th component to the coPROP completion of $H_\bullet(\mathsf{B}^{\circlearrowleft}(\Com))$.
There are two steps that remain. First, we saw in the proof of Theorem \ref{th:main1} that $H^{(r)}_d(\Der^+(\Com(V)),\k)$ is contained in the direct sum of several copies of 
$\Hom(V^{\otimes d},V^{\otimes (r+d)})$, so it is enough to show that 
\[
\left(H^{(r)}_d(\Der^+(\Com(V)),\k)\otimes \Hom(V^{\otimes (r+d)},V^{\otimes d})\right)^{\mathfrak{gl}(V)}=0.
 \]
For these parameters, the stable range condition is given exactly by $n>r+2d$. Second, in the case of the operad $\Com$ a very particular phenomenon presents itself: if we consider it as a wheeled operad, it has just one wheeled relation stating that the universal trace of the operator of multiplication by $a$ from $\partial(\Com)$ vanishes, and this wheeled operad is Koszul (see \cite{MR2483835} or Corollary \ref{cor:barcom}), so the homology $\mathsf{B}^{\circlearrowleft}(\Com)$, as well as the coPROP completion of that homology, is concentrated on the diagonal $r=d$. 
\end{proof}

As we mentioned above, in \cite{MR0725416} the chain complexes of Theorem \ref{th:graphcx1} in the case $\calO=\Com$ are considered, and it is shown that the dimensions of homology of subcomplexes of connected graphs are given by factorials, see \cite[Prop.~1.5]{MR0725416} in the case of a tree and \cite[Prop.~1.6]{MR0725416} in the case of a wheel. This is in perfect agreement with the fact that, according to \cite[Ex.~5.2.5]{MR2483835}, the Koszul dual wheeled operad of $\Com$ is $\Lie^{\circlearrowleft}=\Lie\oplus\Cyc(\mathbbold{1})$; moreover, this latter isomorphism clarifies the difference between the two linear species involved (whose components happen to have the same dimensions).
It is also worth noting that the Koszul property of $\Com$ as a wheeled operad, which is the key ingredient of the proof of Corollary~\ref{cor:newFuchs}, is a very rare coincidence. Indeed, in the case of $\Com$, vanishing of all traces follows from vanishing of traces of all generators, a result which is certainly false for most operads. A class of operads where this statement is true includes, in particular, all operads obtained from weight graded commutative associative algebras \cite{MR3084563,khoroshkin2005}, and it would be interesting to compute the  homology of wheeled bar constructions of some operads of that class, generalizing \cite[Th.~5.2, Th.~5.3]{MR3084563}. \\

From the proof of Corollary \ref{cor:newFuchs}, one can see that we actually proved more: from vanishing of the homology of $\mathsf{B}^{\circlearrowleft}(\calO)$ in a certain range, one can deduce vanishing of $H_\bullet(\Der^+(\calO(V)),\k)$ in a certain range. In particular, the K\"unneth formula implies that to compute $H_1(\Der^+(\calO(V)),\k)$ we only need the first homology of $\mathsf{B}^{\circlearrowleft}(\calO)$, and to compute $H_2(\Der^+(\calO(V)),\k)$, we only need the first and the second homology of $\mathsf{B}^{\circlearrowleft}(\calO)$. Once that computation is done, one may try to convert the information obtained into information for a given number of generators. For example, in the case $\calO=\Com$, Feigin and Fuchs \cite{MR0757265} proved that the Lie algebra 
 \[
L_1(n)=\Der^+(\Com(V))
 \]
is generated by elements of weight $1$ for $\dim(V)>1$, and that the relations between these elements are all of weight $2$ for $\dim(V)>2$. However, their argument is highly specific to the case of the operad of associative commutative algebras. In fact, the following examples of the associative operad and of the operad of Lie algebras show that in general the situation is much more complicated, but our methods can be useful.

\begin{example}
Applying Corollary \ref{cor:barass} for $d=1$, we obtain 
 \[
H_1^{(r)}(\mathsf{B}^{\circlearrowleft}(\Ass))(n)\cong
\begin{cases}
\qquad\qquad \Ass^{\ac}(2),\qquad\qquad\,\,\,\, n=2, r=1,\\
\Cyc(s\mathbbold{1})(1)\oplus \Cyc(s\mathbbold{1})(1), \quad n=1, r=1,\\
s^{-1}(\Cyc(s\mathbbold{1})\otimes\Cyc(s\mathbbold{1}))(2), n=2, r=2,\\
\qquad\qquad\qquad0 \qquad\qquad\qquad\text{otherwise}.
\end{cases}
 \]
Here, in terms of graphs, the first line corresponds to the binary corolla, the next line corresponds to the two different ways to create a loop connecting a leaf of a binary corolla with the root, and the third line corresponds to the loop connecting the middle leaf of a ternary corolla with the root. This fully agrees with the known result \cite[Th.~1.2]{MR3047471} stating that
 \[
H_1^{(r)}(\Der^+(\Ass(V)))\cong \left(V^*\otimes V^{\otimes 2}\right)\oplus V^{\otimes 2}, 
 \]
for $r<\dim(V)$, where the first summand consists of all derivations of degree one, and the second summand is spanned by all derivations induced by the maps 
 \[
l_{v_1}r_{v_2}\in\Hom_{\k}(V,\Ass(V)) \text{ with } v_1,v_2\in V 
 \]
found in \cite[Sec.~6]{MR2508214}. Indeed, the first summand is ``detected'' by 
 \[
\Ass^{\ac}(2)\oplus\Cyc(s\mathbbold{1})(1)\oplus \Cyc(s\mathbbold{1})(1),
 \] 
and the second one by $s^{-1}(\Cyc(s\mathbbold{1})\otimes\Cyc(s\mathbbold{1}))(2)$. 
\end{example}

\begin{example}\label{ex:nonkoszul}
Applying Corollary \ref{cor:barlie} for $d=1$, we obtain
 \[
H_1^{(r)}(\mathsf{B}^{\circlearrowleft}(\Lie))(n)\cong
\begin{cases}
\,\,\Lie^{\ac}(2),\quad\! n=2, r=1,\\
\mathrm{Hook}_1(n), n\ge 1, r=n, \\ 
\quad\quad0 \quad\quad\text{ otherwise},
\end{cases}
 \]
where for each $n$ the module $\mathrm{Hook}_1(n)$ is the trivial one-dimensional $S_n$-module. This fully agrees with the known result \cite[Th.~1.4.11]{kassabov2003automorphism} stating that 
 \[
H_1^{(r)}(\Der^+(\Lie(V)))\cong \left(V^*\otimes \Lambda^2(V)\right)\oplus\bigoplus_{n\ge 2} S^n(V) 
 \]
for $r<\dim(V)(\dim(V)-1)$, where the first summand consists of all derivations of degree one, and the second summand is spanned by all derivations induced by the maps 
 \[
(\mathrm{ad}_v)^n\in\Hom_{\k}(V,\Lie(V)) \text{ with } v\in V. 
 \]
Indeed, the first summand is detected by 
 \[
\Lie^{\ac}(2)\oplus \mathrm{Hook}_1(1),
 \]
and each summand $S^n(V)$ is detected by $\mathrm{Hook}_1(n)$.
\end{example}

\subsection{Stable homology for the Lie algebras of derivations of zero divergence}

In this section, we shall prove our second main result concerning Lie algebras of derivations of zero divergence. 

Let us start with the following general observation. Suppose that $\mathfrak{g}$ is a Lie algebra, $M$ is a $\mathfrak{g}$-module. Let $\phi$ be a $1$-cocycle of $\mathfrak{g}$ with values in $M$. We shall use this data to define a dg Lie algebra 
 \[
\mathfrak{g}\ltimes_\phi M:=(\mathfrak{g}\oplus s^{-1}M,[-,-],d)
 \]
with the Lie bracket given by 
 \[
[g_1+s^{-1}m_1,g_2+s^{-1}m_2]=[g_1,g_2]+s^{-1}(g_1(m_2)-g_2(m_1))
 \] 
and the differential given by 
 \[
d(g+s^{-1}m)=s^{-1}\phi(g).
 \] 
It follows from the $1$-cocycle property that $d$ is a derivation of the bracket, so $\mathfrak{g}\ltimes_\phi M$ is indeed a dg Lie algebra.

Let us note that for an augmented operad the $\Der(\calO(V))$-module $|\partial(\calO)(V)|$ splits as $\k\oplus \overline{|\partial(\calO)(V)|}$, and it follows from Proposition \ref{prop:divcocycle} that we may view the divergence as a $1$-cocycle of $\Der^+(\calO(V))$ with values in $\overline{|\partial(\calO)(V)|}$. Therefore, we may form the dg Lie algebra $\Der^+(\calO(V))\ltimes_{\Div} \overline{|\partial(\calO)(V)|}$.

We shall now prove the following analogue of Theorem \ref{th:graphcx1}.

\begin{theorem}\label{th:graphcx2}
The $\mathfrak{gl}(V)$-invariant Chevalley--Eilenberg complexes 
 \[
C_\bullet(\Der^+(\calO(V))\ltimes_{\Div} \overline{|\partial(\calO)(V)|}, \Hom(V^{\otimes \bullet_1},V^{\otimes \bullet_2}))^{\mathfrak{gl}(V)}
 \]
are stably isomorphic to the coPROP completion of the cobar construction $\mathsf{B}^{\circlearrowleft}(\calO^{\circlearrowleft})$. This isomorphism is natural with respect to operad morphisms.
\end{theorem}

\begin{proof}
As in the proof of Theorem \ref{th:graphcx1}, we note that
\begin{multline*}
C_\bullet(\Der^+(\calO(V))\ltimes_{\Div} \overline{|\partial(\calO)(V)|},\Hom(V^{\otimes p},V^{\otimes q}))\cong\\
S(s\Der^+(\calO(V))\oplus\overline{|\partial(\calO)(V)|})\otimes\Hom(V^{\otimes p},V^{\otimes q})\cong\\  
\bigoplus_{\substack{k\ge 1,\\ n_1,\ldots,n_k,m_1,\ldots,m_l\ge 1\\ p_1,\ldots,p_k,q_1,\ldots,q_l\ge 0}}
\bigotimes_{i=1}^kS^{p_i}(sV^*\otimes\overline{\calO}(n_i)\otimes_{\k S_{n_i}}V^{\otimes n_i})\\
\otimes
\bigotimes_{j=1}^lS^{q_j}(\overline{|\partial(\calO)|}(m_j)\otimes_{\k S_{m_j}}V^{\otimes m_j})
\otimes (V^*)^{\otimes p}\otimes V^{\otimes q}.   
\end{multline*}
Moreover, each summand in that sum is, as a $\mathfrak{gl}(V)$-module, a submodule of several copies of $V^{\otimes N}\otimes (V^*)^{\otimes M}$ where
 \[
N=\sum_{i=1}^kn_ip_i+\sum_{j=1}^lm_jq_j+q, \quad M=\sum_{i=1}^k p_i +p.
 \] 
We also note that our chain complex is bi-graded:
 \[
C^{(r)}_d(\Der^+(\calO(V))\ltimes_{\Div} \overline{|\partial(\calO)(V)|},\Hom(V^{\otimes p},V^{\otimes q}))^{\mathfrak{gl}(V)} 
 \]
consists of elements of weight $r$ (computed out of the weight grading on the Lie algebra $\Der^+(\calO(V))$ and on $\overline{|\partial(\calO)|(V)}$) and of homological degree $d$. We have $r=\sum_{i=1}^k (n_i-1)p_i+\sum_{j=1}^l m_jq_j$ and $d=\sum_{i=1}^k p_i$, so in particular $r+d=\sum_{i=1}^k n_ip_i+\sum_{j=1}^l m_jq_j$. Thus, our above formulas can be written as 
 \[
N=r+d+q, M=d+p. 
 \]
Instead of thinking of sums of tensor products of symmetric powers above, we shall visualize our complex as the vector space spanned by all linear combinations of sets of ``LEGO pieces'' of the following four types: 
\begin{enumerate}
\item a corolla whose vertex carries a label from $\overline{\calO}(\text{out}(v))$, the incoming half-edge is labelled by an element $\xi\in V^*$, and each outgoing half-edge $e$ is labelled by an element $u_e\in V$ (here we implicitly impose an equivalence relation saying that if we we simultaneously act on the label of the vertex and on labels of outgoing half-edges by the same permutation, the corolla does not change); 
\item a ``cul-de-sac vertex'', that is a vertex $v$ with $n_v\ge 2$ outgoing half-edges and no incoming half-edges, where $v$ carries a label from $\overline{|\partial(\calO)|}(\text{out}(v))=\overline{\calO^{\circlearrowleft}_\w}(\text{out}(v))$, and each outgoing half-edge $e$ is labelled by an element $u_e\in V$ (here we implicitly impose an equivalence relation saying that if we we simultaneously act on the label of the vertex and on labels of outgoing half-edges by the same permutation, the corolla does not change); 
\item a source whose vertex uniquely corresponds to an element of $\{1,\ldots,q\}$, and the outgoing half-edge is labelled by an element $u\in V$;   
\item a sink whose vertex $v$ uniquely corresponds to an element of $\{1,\ldots,p\}$, and the incoming half-edge is labelled by an element $\xi\in V^*$.
\end{enumerate}
Corollas with $n_v$ outgoing half-edges represent elements of the vector spaces
$sV^*\otimes\calO(n_v)\otimes_{\k S_{n_v}}V^{\otimes n_v}$; they have homological degree $1$, and so reordering them creates a sign. Cul-de-sac vertices with $n_v$ outgoing half-edges represent elements of the vector spaces $\overline{|\partial(\calO)|}(n_v)\otimes_{\k S_{n_v}}V^{\otimes n_v}$; they have homological degree $0$, and so reordering them does not create any signs. The sources and the sinks appear in the natural order of their vertex labels, reproducing the vector space $(V^*)^{\otimes p}\otimes V^{\otimes q}$.
 
From the First and the Second Fundamental Theorems for $\mathfrak{gl}(V)$, it follows that to have non-zero invariants at all we must have $N=M$, and that for the ``stable range'' $\dim(V)\ge M=d+p$ the $\mathfrak{gl}(V)$-invariants in this module are isomorphic to a vector space with a basis of (equivalence classes of) graphs obtained by assembling the LEGO pieces together, that is directed decorated graphs $\Gamma$ obtained by matching all the incoming half-edges with all the outgoing half-edges. Specifically, we obtain that each vertex $v$ of $\Gamma$ is one of the following three types:
\begin{enumerate}
\item a corolla with one incoming half-edge and $n_v\ge 1$ outgoing half-edges labelled by an element from $\overline{\calO}(\text{out}(v))$;   
\item a cul-de-sac vertex with $n_v\ge 1$ outgoing half-edges labelled by an element from $\overline{\calO^{\circlearrowleft}_\w}(\text{out}(v))$;
\item a source with one outgoing half-edge uniquely corresponding to an element of $\{1,\ldots,q\}$;    
\item a sink with one incoming half-edge uniquely corresponding to an element of $\{1,\ldots,p\}$.  
\end{enumerate}
Two such graphs are equivalent if there is an isomorphism of them that agrees with all labels. As above, corollas have homological degree $1$, and so reordering them creates a sign, and cul-de-sac vertices have homological degree zero, and so reordering them does not create any signs. Such a graph corresponds to an invariant that pairs the copies of $V$ and $V^*$ according to the edges between the vertices in the graph. 

Similarly to Lemma \ref{lm:fuchs}, we see that connected graphs that appear as basis elements in the stable range are precisely the connected graphs appearing as basis elements of the wheeled bar construction $\mathsf{B}^{\circlearrowleft}(\calO^{\circlearrowleft})$. Indeed, our corollas and cul-de-sac vertices are precisely the building blocks of the monad of rooted trees and wheels. Moreover, the labels of the corollas are taken from $\overline{\calO}=\overline{\calO^{\circlearrowleft}_\o}$, the labels of cul-de-sac vertices are taken from $\overline{|\partial(\calO)|}=\overline{\calO^{\circlearrowleft}_\w}$, and the signs arising from reordering them are in agreement with those coming from unordered tensor products in the definition of $\mathsf{B}^{\circlearrowleft}(\calO^{\circlearrowleft})=\mathsf{T}^{\circlearrowleft}(s\overline{\calO^{\circlearrowleft}_\o}\oplus \overline{\calO^{\circlearrowleft}_\w})$.

Finally, the differential 
\begin{multline*}
C^{(w)}_{d+1}(\Der^+(\calO(V))\ltimes_{\Div}\overline{|\partial(\calO)|(V)},\Hom(V^{\otimes p},V^{\otimes q}))^{\mathfrak{gl}(V)}\to \\ 
C^{(w)}_d(\Der^+(\calO(V))\ltimes_{\Div}\overline{|\partial(\calO)|(V)},\Hom(V^{\otimes p},V^{\otimes q}))^{\mathfrak{gl}(V)} 
\end{multline*}
admits an elegant formula in the stable range $\dim(V)>d+p$, where the strict inequality comes from the fact that we need to know basis elements in homological degree $d+1$ as well. As before, we note that since the action of $\Der(\calO(V))$ on $\mathfrak{gl}(V)$-modules is via the augmentation, the subalgebra $\Der^+(\calO(V))$ acts on all such modules trivially, and so the Chevalley--Eilenberg differential only uses the Lie bracket and the differential in $\Der^+(\calO(V))\ltimes_{\Div}\overline{|\partial(\calO)|(V)}$. Similarly to how the Lie algebra of derivations comes from a pre-Lie structure, the action of derivations on $|\partial(\calO)|(V)$ can be obtained by simultaneously using the right action of $\calO$ on $|\partial(\calO)|$ and computing the pairing of $V^*$ with $V$. Also, by the definition of divergence, it is computed by simultaneously projecting onto the commutator quotient and computing the pairing of $V^*$ with $V$. As in the proof of Theorem \ref{th:graphcx1}, for the complex of invariants, this means that the computation of each Lie bracket in $\Der^+(\calO(V))\ltimes_{\Div}\overline{|\partial(\calO)|(V)}$ contracts appropriate non-loop edges of graphs, and the computation of the differential contracts loops in all possible ways. (As before, 
one way to check that the signs match would be to dualize and consider the Chevalley--Eilenberg cohomological complex and the cobar construction; for each of them the differential is a derivation and hence is ``local'', that is applies to each corolla individually, and the fact that the two differentials agree on corollas is almost obvious.) 

Passing to disconnected graphs corresponds to computing tensor and symmetric powers with respect to the Cauchy product, which is precisely the coPROP completion of the wheeled bar construction viewed as a wheeled cooperad. The statement on naturality is obvious.
\end{proof}

We shall now explain how to use that result for the homology of the Lie algebra of divergence-free derivations. 

\begin{theorem}\label{th:main2}
Let $\calO$ be an augmented operad, and $V$ a finite-dimensional vector space. The vector spaces
 \[
H_\bullet(\SDer^+(\calO(V)),\Hom(V^{\otimes \bullet_1},V^{\otimes \bullet_2}))^{\mathfrak{gl}(V)}
 \] 
stabilize as $\dim(V)\to\infty$, with the stable value given by the coPROP completion of the wheeled cooperad $H_\bullet(\mathsf{B}^{\circlearrowleft}(\calO^{\circlearrowleft}))$. This isomorphism is natural with respect to operad morphisms.
\end{theorem}

\begin{proof}
Let us briefly go back to the general case of the Lie algebra $\mathfrak{g}\ltimes_\phi M$. From the $1$-cocycle equation, it follows that $\mathfrak{g}^\phi:=\ker\phi$ is a Lie subalgebra of $\mathfrak{g}$. In our next result, homology of that Lie subalgebra appears (heuristically, this result corresponds to the fact that $\mathfrak{g}^\phi\hookrightarrow \mathfrak{g}\ltimes_\phi M$ is a weak equivalence of dg Lie algebras, but we spell out an elementary proof for the convenience of the reader). 

\begin{lemma}\label{lm:derived-div}
Suppose that $\phi$ is a surjective map $\mathfrak{g}\to M$. We have
 \[
H_\bullet(\mathfrak{g}\ltimes_\phi M,\k)\cong H_\bullet(\mathfrak{g}^\phi,\k).
 \]
\end{lemma}

\begin{proof}
The Chevalley--Eilenberg differential of $C_\bullet(\mathfrak{g}\ltimes_\phi M,\k)$ splits as a sum of the ``horizontal part''
 \[
d_h\colon C_i(\mathfrak{g}\ltimes_\phi M,\k)\to C_{i-1}(\mathfrak{g}\ltimes_\phi M,\k)
 \] 
using the Lie algebra structure only and the ``vertical part''
 \[
d_v\colon C_i(\mathfrak{g}\ltimes_\phi M,\k)\to C_{i-1}(\mathfrak{g}\ltimes_\phi M,\k)
 \] 
using the differential only. Those maps are anticommuting differentials; moreover, if one places $S^i(s\mathfrak{g})\otimes S^j(M)\subset S(s(g\oplus s^{-1}M))=C_\bullet(\mathfrak{g}\ltimes_\phi M,\k)$ in bi-degree $(i+j,-j)$, these maps make $C_\bullet(\mathfrak{g}\ltimes_\phi M,\k)$ into a homological double complex $C_{\bullet,\bullet}$ concentrated in the fourth quadrant. For such a double complex, the column filtration on the direct sum total complex is bounded below and exhaustive, so by the Classical Convergence Theorem \cite[Th.~5.5.1]{MR1269324} the corresponding spectral sequence $\{{}^IE^r_{p,q}\}$ converges to $H_\bullet(\mathrm{Tot}^{\oplus}C_{\bullet,\bullet})$, which is precisely the homology of the Chevalley--Eilenberg complex we are interested in.    

By choosing a section of $\phi$ (and thus identifying $\mathfrak{g}$ with $\mathfrak{g}^\phi\oplus M$ as vector spaces), we may identify the vertical differential $d_h$ with $\id\otimes d_K$ on $S(s\mathfrak{g}^\phi)\otimes S(sM\oplus M)$ with $d_K$ being the standard Koszul differential on $S(sM\oplus M)$, so 
we have 
 \[
{}^IE^1_{p,q}=
\begin{cases}
C_p(\mathfrak{g}^\phi,\k), \quad q=0,\\
\qquad 0, \qquad  \,\,\, q\ne 0,
\end{cases}
 \]
which implies that our spectral sequence abuts on the second page, giving us the isomorphism
$H_\bullet(\mathfrak{g}\ltimes_\phi M,\k)\cong H_\bullet(\mathfrak{g}^\phi,\k)$ 
that we wish to establish. 
\end{proof} 

We wish to apply this result to the case of $\mathfrak{g}=\Der^+(\calO(V))$, $M=\overline{|\partial(\calO)(V)|}$, the augmentation ideal part of the commutator quotient, and $\phi$ the divergence. On elements of $\mathfrak{g}$, the divergence does indeed take values in $M$, and moreover we clearly have $\mathfrak{g}^\phi=\SDer^+(\calO(V))$. However, at a first glance Lemma \ref{lm:derived-div} is not directly applicable since the divergence is not in general surjective. Nevertheless (and that is sufficient for our purposes), it is stably surjective \cite[Prop.~8.20]{Powell21}. For convenience, we prove the following version of stable surjectivity.

\begin{lemma}
The divergence $\Div\colon \Der^+(\calO(V))\to \overline{|\partial(\calO)(V)|}$ is surjective on elements of weight strictly less than $\dim(V)$.
\end{lemma}

\begin{proof}
Suppose that $\dim(V)=n$, and choose a basis $x_1,\ldots,x_n$ of $V$. Let $u\in \overline{|\partial(\calO)(V)|}$ be an element of weight $w<n$. Note that we can write $u=u^{(1)}+\cdots+u^{(n)}$, where $u^{(i)}$ does not depend on $x_i$. To represent the element $u^{(i)}$ as the divergence of some derivation, we lift it to an element $f^{(i)}$ of the same weight in $\overline{\partial(\calO)(V)}$ arbitrarily, and note that $f^{(i)}$ can be written in a nearly tautological way as $\frac{\partial g^{(i)}}{\partial x_i}$ for an element $g^{(i)}\in\overline{\calO}(V)$ of degree one in $x_i$. It follows that the derivation $D_i$ of $\calO(V)$ that sends $x_i$ to $g^{(i)}$ and all other generators to zero satisfies $\Div(D_i)=u^{(i)}$, and so we have $\Div(D_1+\cdots+D_n)=u$, as required.  
\end{proof}

Thus, we have the following isomorphisms of stable limits:
\begin{multline*}
\lim_{\dim(V)\to\infty}\left(H_\bullet(\SDer^+(\calO(V)),\Hom(V^{\otimes p},V^{\otimes q}))\right)^{\mathfrak{gl}(V)}\cong\\
\lim_{\dim(V)\to\infty}\left(H_\bullet(\SDer^+(\calO(V)),\k)\otimes \Hom(V^{\otimes p},V^{\otimes q})\right)^{\mathfrak{gl}(V)}\cong\\
\lim_{\dim(V)\to\infty}\left(H_\bullet(\Der^+(\calO(V))\ltimes_{\Div} \overline{|\partial(\calO)(V)|},\k)\otimes \Hom(V^{\otimes p},V^{\otimes q})\right)^{\mathfrak{gl}(V)}\cong\\
\lim_{\dim(V)\to\infty}\left(H_\bullet(\Der^+(\calO(V))\ltimes_{\Div} \overline{|\partial(\calO)(V)|}, \Hom(V^{\otimes p},V^{\otimes q}))\right)^{\mathfrak{gl}(V)}.
\end{multline*}
Here the first and the last isomophism come from the fact that the action of our algebras on the $\mathfrak{gl}(V)$-module $\Hom(V^{\otimes p},V^{\otimes q})$ is trivial, and the middle isomorphism follows from Lemma \ref{lm:derived-div} and the stable surjectivity of divergence (for each fixed weight $w$ of the homology, Lemma \ref{lm:derived-div} requires surjectivity of divergence on elements of weight at most $w$). 

To complete the proof, we argue as in the proof of Theorem~\ref{th:main1}: for each $w,d,p,q$, then as long as $\dim(V)>d+p$, the components of weight $w$ and homological degrees $d$ and $d+1$ of the complex 
 \[
C_\bullet(\Der^+(\calO(V))\ltimes_{\Div} \overline{|\partial(\calO)(V)|}, \Hom(V^{\otimes \bullet_1},V^{\otimes \bullet_2}))^{\mathfrak{gl}(V)}
 \] 
are nonzero only for $p=w+q$ and are isomorphic to the corresponding components of the wheeled bar construction $\mathsf{B}^{\circlearrowleft}(\calO^{\circlearrowleft})$; moreover, the differential of the Chevalley--Eilenberg complex is identified under that isomorphism with the differential of the wheeled bar construction. Since computing homology commutes with taking $\mathfrak{gl}(V)$-invariants, we conclude that for such $w,d,p,q$, we have 
 \[
H^{(w)}_d(\Der^+(\calO(V))\ltimes_{\Div} \overline{|\partial(\calO)(V)|}, \Hom(V^{\otimes \bullet_1},V^{\otimes \bullet_2}))^{\mathfrak{gl}(V)}\cong H^{(w)}_d(\mathsf{B}^{\circlearrowleft}(\calO^{\circlearrowleft})).
 \] 
It remains to apply the K\"unneth formula to obtain the desired result for the coPROP completion.
\end{proof}

Let us also record an analogue of Proposition \ref{prop:thmain1}. In order to state a result using invariants of $\mathfrak{gl}(V)$, we shall use an auxiliary Lie algebra $\SDer^\wedge(\calO(V))$ of derivations with constant divergence; that Lie algebra included in the exact sequence
 \[
0\to\SDer^+(\calO(V))\to \SDer^\wedge(\calO(V))\to \mathfrak{gl}(V)\to 0.
 \]

\begin{proposition}\label{prop:thmain2}
Let $\calO$ be an augmented operad, and $V$ a finite-dimensional vector space. The homology 
 \[
H_\bullet(\SDer^\wedge(\calO(V)),\Hom(V^{\otimes p},V^{\otimes q}))
 \] 
is isomorphic to the cofree $H_\bullet(\mathfrak{gl}(\k))$-comodule cogenerated by 
 \[
H_\bullet(\SDer^+(\calO(V)),\Hom(V^{\otimes p},V^{\otimes q}))^{\mathfrak{gl}(V)}. 
 \] 
The latter is isomorphic to  
 \[
(H_\bullet(\SDer^+(\calO(V)),\k)\otimes\Hom(V^{\otimes p},V^{\otimes q}))^{\mathfrak{gl}(V)}.
 \] 
\end{proposition}

This result is proved in the exact same way as Proposition \ref{prop:thmain1}, so we omit the proof. 
\\

The following result is well known, see, for example \cite[Exercice 10.2.1]{MR1600246}. It was one of the inspirations for Theorem \ref{th:main2}, but we can also use that theorem to re-prove it.

\begin{corollary}
Let $A$ be a unital $\k$-algebra, and let $\mathfrak{sl}_n(A)$ to be the kernel of the map $\mathfrak{gl}_n(A)\to A/[A,A]$ sending $f$ to the image of $\tr(f)$ in the commutator quotient. If one defines $\mathfrak{sl}(A)=\varinjlim\mathfrak{sl}_n(A)$, we have a Hopf algebra isomorphism
 \[
H_\bullet(\mathfrak{sl}(A))\cong S^c(\overline{HC}_{\bullet-1}(A)),     
 \]
where $\overline{HC}$ is the reduced cyclic homology (without $HC_0(A)=A/[A,A]$). 
\end{corollary}

\begin{proof}
Like in Corollary \ref{cor:LQT}, we consider the algebra $A_+$ which is obtained from~$A$ by adjoining a unit (even though $A$ is already unital), and view $A_+$ as an operad~$\calO$ concentrated in arity one. According to Example \ref{ex:divforalg}, for a vector space $V$ of dimension $n$, the Lie algebra $\mathfrak{sl}_n(A)$ is exactly $\SDer^+(\calO(V))$, so 
\begin{multline*}
(C_\bullet(\SDer^+(\calO(V)),\k)\otimes\Hom(V^{\otimes p},V^{\otimes q}))^{\mathfrak{gl}(V)}\cong\\ (C_\bullet(\mathfrak{sl}_n(A),\k)\otimes\Hom(V^{\otimes p},V^{\otimes q}))^{\mathfrak{gl}(V)}.
\end{multline*}
Theorem \ref{th:main2} implies that the stable limit of 
 \[
(C_\bullet(\mathfrak{sl}_n(A),\k)\otimes\Hom(V^{\otimes p},V^{\otimes q}))^{\mathfrak{gl}(V)}
 \]
for $\dim(V)\to\infty$ is the coPROP completion of $\mathsf{B}^{\circlearrowleft}((A_+)^{\circlearrowleft})$.
As we saw in the proof of Theorem \ref{th:main2}, the operadic part of the latter complex is the bar construction $\mathsf{B}(A_+)$, which is acyclic since the augmentation ideal of $A_+$ is a unital algebra, and the wheeled part is the loop-less part of the cyclic bar construction, which precisely computes the (shifted) reduced cyclic homology. Since once again the homology is concentrated in the $(0,0)$-component of our wheeled cooperad, the coPROP completion is just the cofree cocommutative coalgebra $S^c(\overline{HC}_{\bullet-1}(A))$, and the proof is completed in the same way as for Corollary \ref{cor:LQT}.
\end{proof}

Note that the results of Powell \cite[Th.~12.1]{Powell21} can be used to deduce that the stable homology of $\SDer^+(\calO(V))$ in degree one is concentrated in weight one for any binary operad $\calO$. 
Theorem \ref{th:main2} allows us to prove one general statement for which one generally has no hope if considering all derivations, see Example~\ref{ex:nonkoszul}. To state it, recall that a weight graded Lie algebra $\mathfrak{g}$ is said to be Koszul in weight $w$ if its homology in weight $w$ is concentrated in homological degree $w$. 

\begin{corollary}\label{cor:Koszul}
The wheeled operad $\calO^{\circlearrowleft}$ is Koszul in weights $\le r_0$ if and only the Lie algebra $\SDer^+(\calO(V))$ is Koszul in all weights $r<\min(r_0,\frac13\dim(V))$ for every finite-dimensional vector space $V$.

In particular, the wheeled operad $\calO^{\circlearrowleft}$ is Koszul if and only if the Lie algebra $\SDer^+(\calO(V))$ is Koszul in all weights $r<\frac13\dim(V)$ for every finite-dimensional vector space $V$.
\end{corollary}

\begin{proof}
Like in the proof of Corollary \ref{cor:newFuchs}, one sees that the dimension range $\dim(V)>r+2d$ is where we can use the coPROP completion of the bar construction to infer information about the weight $r$ part of the degree $d$ homology of the Lie algebra $\SDer^+(\calO(V))$. Moreover, the argument of that proof shows that if $\calO^{\circlearrowleft}$ is Koszul in weights $\le r_0$, then, in that dimension range, we have 
 \[
H^{(r)}_d(\SDer^+(\calO(V)),\k)=0 \text{ for } d\ne r\le r_0.
 \] 
Conversely, if $H^{(r)}_d(\SDer^+(\calO(V)),\k)=0$ for $d\ne r\le r_0$ and $\dim(V)>r+2d$, then, in particular, the homology of $\mathsf{B}^{\circlearrowleft}(\calO^{\circlearrowleft})$ vanishes for $d\ne r\le r_0$, since those weight graded components appear in the coPROP completion from Theorem~\ref{th:main2}. Finally, we always have $r\ge d$, so if $\dim(V)> 3r$, we automatically have $\dim(V)>r+2d$, completing the proof of the first statement. The second statement is a trivial consequence of the first one, since a wheeled operad is Koszul if and only if it is Koszul in weights $\le r_0$ for all $r_0$.
\end{proof}

Let us recall some examples where Corollary \ref{cor:Koszul} applies. 

\begin{example}\leavevmode 
\begin{enumerate}
\item For $\calO=\Com$, we have \cite[Th.~7.1.2]{MR2483835}
 \[
H_\bullet(\mathsf{B}^{\circlearrowleft}(\Com^{\circlearrowleft})_\o)\cong \Com^{\ac},\quad H_\bullet(\mathsf{B}^{\circlearrowleft}(\Com^{\circlearrowleft})_\w)\cong\overline{\Cyc}(s\mathbbold{1}),
 \]
where $\overline{\Cyc}$ corresponds to removing elements of degree one,
\item for $\calO=\Ass$, we have \cite[Th.~A]{MR2483835}
 \[
H_\bullet(\mathsf{B}^{\circlearrowleft}(\Ass^{\circlearrowleft})_\o)\cong \Ass^{\ac},\quad H_\bullet(\mathsf{B}^{\circlearrowleft}(\Ass^{\circlearrowleft})_\w)\cong\Cyc(s\mathbbold{1})\otimes\Cyc(s\mathbbold{1}),
 \]
\item for $\calO=\Lie$, we have \cite[Th.~4.1.1]{MR2641194} 
 \[
H_\bullet(\mathsf{B}^{\circlearrowleft}(\Lie^{\circlearrowleft})_\o)\cong \Lie^{\ac},\quad H_\bullet(\mathsf{B}^{\circlearrowleft}(\Lie^{\circlearrowleft})_\w)=0.
 \]
\end{enumerate}
Note that the first statement bears a striking resemblance with Corollary \ref{cor:barcom}, the second is somewhat related to Corollary \ref{cor:barass}, while the third one is drastically different from what Corollary \ref{cor:barlie} may suggest. It is not entirely clear how to interpret this, though one natural speculation is to examine the role of the operadic module of K\"ahler differentials, and not just the module $\partial(\calO)$ corresponding to universal multiplicative enveloping algebras. 
\end{example}

\begin{corollary}
For $\calO=\Com, \Ass, \Lie$, the Lie algebra $\SDer^+(\calO(V))$ is stably Koszul.
\end{corollary}

We should warn the reader that it is not enough to require that an operad $\calO$ is Koszul to conclude that the wheeled operad $\calO^{\circlearrowleft}$ is Koszul: the above-mentioned result of Bruinsma and Khoroshkin implies that it is not true for the operad of Poisson algebras. (An intriguing question that we do not consider in this paper is to apply our results to compute homology of Lie algebras of derivations of free algebras over operads of topological origin, such as the operad of Gerstenhaber algebras.) \\

A very interesting but probably highly nontrivial question is to determine, for the given fixed (sufficiently large) value of $\dim(V)$, the first and the second homology of the Lie algebra $\SDer^+(\calO(V))$, that is, the minimal sets of generators and relations of those algebras. It is known \cite{MR0607583} that the Lie algebra $\SDer^+(\Com(V))$ is always generated by elements of degree one. Let us give two simple examples showing that this is way too much to hope for in general. 

\begin{example}
Let us consider the free Lie algebra $\Lie(V)$ with $V=\mathrm{Vect}(x,y)$. Let us note that each derivation of degree one of that Lie algebra sends both $x$ and~$y$ to a multiple of $[x,y]$, and a direct calculation shows that a nonzero derivation of that form cannot have zero divergence. Thus, to show that the Lie algebra $\SDer^+(\Lie(V))$ is not generated by elements of degree one, it is enough to show that derivations with zero divergence exist. For that, we consider the inner derivation $D=\mathrm{ad}_{[x,y]}$. We have 
\begin{gather*}
D(x)=[[x,y],x]=[x,y]x-x(xy-yx)=([x,y]+xy)x-x^2y,\\
D(y)=[[x,y],y]=[x,y]y-y[x,y]=y^2x+([x,y]-yx)y,
\end{gather*}
and therefore
 \[
\Div(D)=[x,y]+xy+[x,y]-yx=3[x,y]
 \]
vanishes in the commutator quotient of the universal enveloping algebra. It is also possible to show, through a slightly more involved argument, that $D=\mathrm{ad}_{[x,y]}$ is not a commutator of elements of degree one in the free Lie algebra $\Lie(V)$ with $V=\mathrm{Vect}(x,y,z)$.
\end{example}

\begin{example}
As we saw in Example \ref{ex:bergman}, if we consider the free associative algebra $\Ass(V)$ with $V=\mathrm{Vect}(x,y)$, the derivation $D\in \Der(\Ass(V))$ that sends the generator $x$ to $[x,y]^2$ and the generator $y$ to zero has zero divergence. A direct calculation shows that for a basis of the vector space of elements of degree one of the Lie algebra $\SDer^+(\Ass(V))$ one can take
 \[
x^2\partial_y, y^2\partial_x, x^2\partial_x-(xy+yx)\partial_y, y^2\partial_y-(xy+yx)\partial_x
 \] 
Now we note that each of these derivations belongs to the Lie algebra of symplectic derivations \cite{MR1247289}, that is, they all annihilate the element $[x,y]$. However, the element $D$ sends $[x,y]$ to $[[x,y]^2,y]\ne 0$. Thus, the Lie algebra $\SDer^+(\Ass(V))$ is not generated by elements of degree one.
\end{example}


\section{Discussion of results and perspectives}\label{sec:perspectives}

\subsection{Mixed representation stability}

The notion of mixed representation stability \cite{MR3084430} goes back to the work of Brylinski \cite{MR1035220}, Feigin and Tsygan \cite[Chapter~4]{MR0923136} and Hanlon \cite{MR0946439}. Let us recall the necessary basics that apply in our context. Consider a sequence of $\mathfrak{gl}_n$-modules $V_n$ equipped with maps $V_n\to V_{n+1}$ that are injective for large $n$ and consistent with the embeddings $\mathfrak{gl}_n\hookrightarrow\mathfrak{gl}_{n+1}$. We further assume that the image of $V_n$ in $V_{n+1}$ generates all of $V_{n+1}$ under the action of $\mathfrak{gl}_{n+1}$ for large enough $n$. This sequence is said to be \emph{mixed representation stable} if for all partitions $\alpha,\beta$, the multiplicity of the irreducible $\mathfrak{gl}_n$-module $V(\alpha,\beta)$ with highest weight 
 \[
\sum_i\alpha_ie_i-\sum_j\beta_je_{n+1-j} 
 \]
in $V_n$ is eventually constant. This sequence is said to be \emph{uniformly mixed representation stable} if all multiplicities become eventually constant simultaneously: there is some $N$ such that for all partitions $\alpha,\beta$, the above multiplicity does not depend on $n$ for $n\ge N$. 

Feigin and Tsygan and, independently, Hanlon proved that for a (possibly non-unital) associative algebra $A$ and for each $i\ge 0$, the homology group $H_i(\mathfrak{gl}_n(A),\k)$ is uniformly mixed representation stable, and expressed the stable multiplicity of  $V(\alpha,\beta)$
in $H_\bullet(\mathfrak{gl}_n(A),\k)$ in terms of the cyclic homology and the bar homology of $A$. Our work leads to analogous results for the 
$\mathfrak{gl}_n$-module structures on $H_i(\Der^+(\calO(x_1,\ldots,x_n)),\k)$ and $H_i(\SDer^+(\calO(x_1,\ldots,x_n)),\k)$. Moreover, the language of wheeled operads and PROPs allows us to write a very short and compact formula for multiplicities of irreducible modules as follows.

\begin{theorem}\label{th:repstab}
For each $i\ge 0$, the homology groups $H_i(\Der^+(\calO(x_1,\ldots,x_n)),\k)$ and $H_i(\SDer^+(\calO(x_1,\ldots,x_n)),\k)$ are uniformly mixed representation stable. More precisely, for $n\to\infty$, and for any two partitions $\alpha\vdash p,\beta\vdash q$:
\begin{itemize}
\item the multiplicity of $V(\alpha,\beta)$ in $H_\bullet(\Der^+(\calO(x_1,\ldots,x_n)),\k)$  stabilizes to 
 \[
\left(\calP^c(\mathsf{B}^\circlearrowleft(\calO))(q,p)\otimes S^\alpha\otimes S^\beta\right)_{S_q\times S_p}
 \]
\item the multiplicity of $V(\alpha,\beta)$ in $H_\bullet(\SDer^+(\calO(x_1,\ldots,x_n)),\k)$ stabilizes to 
 \[
\left(\calP^c(\mathsf{B}^\circlearrowleft(\calO^\circlearrowleft))(q,p)\otimes S^\alpha\otimes S^\beta\right)_{S_q\times S_p}
 \]
\end{itemize}
\end{theorem}

Unwrapping the formulas for the Cauchy products involved in the definition of the coPROP completion, one can express the above via the coinvariants of the action on sums of tensor and symmetric products of components of the bar constructions of subgroups $S_{\mathbf{m},\mathbf{k}}\subset S_{m_1+\cdots+m_q+k_1+\cdots+k_r}$, each such subgroup being the symmetry group of the set partition of a set of $m_1+\cdots+m_q+k_1+\cdots+k_r$ elements into blocks of sizes $m_1$,\ldots, $m_q$, $k_1$, \ldots, $k_r$, where the first $q$ blocks are additionally labelled $1,\ldots,q$ in some order. However, that latter formula is much harder to digest than the formula we chose to present.

\subsection{Stable cohomology of automorphism groups of free groups with bivariant coefficients}

Recall that wheeled operads appeared in recent work on the homology of automorphism groups of free groups, specifically, in the computation of the stable limit of  
 \[
H^\bullet(\Aut(F_n),\Hom_{\mathbb{Q}}(H(n)^{\otimes p},H(n)^{\otimes q})),
 \]
where $H(n)=H_1(F_n,\mathbb{Q})$ is the standard representation of $\Aut(F_n)$; in this case, homological stability follows from general results of Randal-Williams and Wahl~\cite{MR3689750} and Randal-Williams~\cite{MR3782426}. A formula for that stable limit was conjectured by Djament \cite{MR3993804} and proved by Lindell \cite{Lin22}. 

Importantly for us, Kawazumi and Vespa \cite{MR4647673} reformulated the conjecture of Djament using wheeled operads; combined with a theorem of Gran\aa{}ker~\cite{Gran08}, their formulation expresses the stable limit as the PROP completion of the linear dual of $H_\bullet(\mathsf{B}^{\circlearrowleft}(\ULie))$, the homology of the wheeled bar construction of the wheeled operad of unimodular Lie algebras. The similarity of this statement and our main results makes one wonder if the two can be related to each other, perhaps using the Johnson homomorphism \cite{MR3497296} that offers a bridge between automorphisms of free groups and derivations of free Lie algebras; however, \emph{unimodular} Lie algebras do not naturally appear in our context, so this question is very far from obvious.  

It is also worth noting that the limit of $H^\bullet(\Aut(F_n),\Hom_{\mathbb{Q}}(H(n)^{\otimes p},H(n)^{\otimes q}))$ can be used to obtain some information on the much more complicated stable cohomology of the so called IA-automorphism groups~\cite{habiro2023stable} ( incidentally, it is possible that conjectures of Katada \cite{katada2022stable} on the Albanese quotient of the homology of IA-automorphism groups may also benefit from being considered from the point of view of wheeled operads). Our strategy for studying the homology groups $H_i(\Der^+(\calO(x_1,\ldots,x_n)),\k)$ and $H_i(\SDer^+(\calO(x_1,\ldots,x_n)),\k)$ follows the very same logic. 

\subsection{Divergence of derivations and tame automorphisms}

Let us mention that the notion of divergence of a derivation of a free algebra has been implicit in the literature on automorphisms of free algebras for several decades. This goes back to the remarkable work of Bryant and Drensky \cite{MR1242836}, see also \cite{MR1356889,MR1242837}. The central notion motivating that work is that of a tame automorphism. Recall that an automorphism of a free algebra $\calO(x_1,\ldots,x_n)$ is called tame if it is obtained by compositions of automorphisms from $GL_n(\k)$ and obvious ``triangular'' automorphisms that send one of the generators $x_i$ to an element of the subalgebra generated by $x_1$,\ldots,$x_{i-1}$,$x_{i+1}$,\ldots,$x_n$ and fixes all those other generators. Already in the case of free commutative associative algebras and free associative algebras non-tame automorphisms exist \cite{MR2015334,MR2338130}, which is one of the factors which make studying automorphisms of free algebras hard. However, for commutative associative algebras, it was proved independently by Shafarevich~\cite{MR0607583} and Anick \cite{MR0704764} that non-tame automorphisms can be topologically approximated by tame ones; in fact, Shafarevich deduces this from the fact that the Lie algebra of the group of all automorphisms consists of all derivations of constant divergence. For some operads, it is possible to show that all tangent derivations of automorphisms must have constant divergence, which identifies a nontrivial class of situations where an analogue of the Anick--Shafarevich theorem may hold. It is interesting to note that groups of tame automorphisms of free operadic algebras emerged in recent work of Bartholdi and Kassabov \cite{bartholdi2023property} who used them to construct groups with Kazhdan's property (T) that map onto large powers of alternating groups. In further work, we intend to explore implications of our results for automorphisms and derivations of free Lie algebras and free associative algebras (where only partial results are presently available, see \cite{MR2508214,MR3047471}), as well as for an analogue of the Anick--Shafarevich approximation theorem in the case of free associative algebras.

\subsection{K-theory of operadic algebras}

A somewhat more conceptual relationship of our work to the study of automorphism groups of free algebras is as follows. Primitive elements of the canonical Hopf algebra structure on the homology of the Lie algebra $\mathfrak{gl}(A)$, given by the cyclic homology of $A$, should be thought of as the ``additive $K$-theory'' of the algebra~$A$, an approximation to the actual $K$-theory of $A$, which, over $\mathbb{Q}$, can be identified with the primitive elements for the canonical Hopf algebra structure on the homology of the group $GL(A)=\varinjlim GL_n(A)$. Similarly, our results suggest that the wheeled bar construction of an operad is an object of additive $K$-theoretic nature that approximates the actual $K$-theory of $\calO$, the primitive elements for the canonical Hopf algebra structure on the homology of the direct limit of $\Aut(\calO(V))$ as $\dim(V)$ goes to infinity. Though these $K$-groups can be described in the general context of algebraic $K$-theory \cite{MR0249491}, automorphism groups of free algebras are considerably more complicated than general linear groups, so this kind of $K$-theory does not seem to have been studied much. As one notable exception, one should mention the work of Connell concerning the $K$-theory of projective commutative associative algebras \cite{MR0384803,MR0409475,MR0556157} which is as exciting as it is demoralizing, for one quickly realizes that studying this subject instantly brings one to notoriously difficult open questions, such as the Zariski Cancellation Problem~\cite{MR0041480} and the Jacobian Conjecture \cite{MR0620428}. If one considers only the group $K_0$, there are more results available for various operads~$\calO$, for instance, in the work of Artamonov \cite{MR0498856} on metabelian Lie algebras, as well as Artamonov~\cite{MR0294213} and Pirashvili~\cite{MR1715892} on  nilpotent algebras in general. One possible way to build a connection between our work and the operadic $K$-theory computed via the automorphism groups would go along the lines of the beautiful theorem of Hennion~\cite{MR4237022} asserting that the cyclic homology of an algebra is the tangent version of its $K$-theory in a very precise sense. (Note that the fact that we work over a field ensures that the notion of a derivation is available at all, contrary, for example, to the cartesian context of Lawvere theories, where some recent significant breakthroughs have been made in studying stable (co)homology of automorphism groups as well as in computing the $K$-groups, see e.g. the work of Szymik and Wahl \cite{MR3953508} and Bohmann and Szymik~\cite{BoSz20} for some compelling examples.)
It is also very natural to attempt a generalization of remarkable recent results on derived representation schemes and cyclic homology \cite{MR3692888,MR3204869,MR3507922}, where one can speculate that such a generalization would replace the derived representation scheme of representations of an algebra $A$ on a vector space $W$ by the derived moduli stack of $\calO$-algebra structures on $W$. We hope to address this circle of questions elsewhere.

\section*{Acknowledgements } This work was supported by Institut Universitaire de France and by the French national research agency [grant ANR-20-CE40-0016]. It started in July 2023 during my visit to the Max Planck Institute for Mathematics in Bonn; I am grateful to MPIM for hospitality and excellent working conditions. During that visit, discussions with Ualbai Umirbaev were absolutely invaluable; in particular, it was from him that I learned about divergence-related techniques used to prove non-tameness of automorphisms, and that was a crucial hint for the right questions to ask in this context, which then unravelled the main results of this work. I am also grateful to Benjamin Enriquez, Mikhail Kapranov, Anton Khoroshkin, Erik Lindell, Sergei Merkulov, Vadim Schechtman, Sergei Shadrin, Arthur Souli\'e, Pedro Tamaroff, and Christine Vespa for very useful discussions of this work at various stages of its preparation. Special thanks are due to Geoffrey Powell who offered many useful comments on preliminary versions of the paper. I also wish to thank the referee of the paper for many remarks and suggestions that substantially improved the exposition, and in particular, led me to discover the very neat Theorem~\ref{th:graphcx2}. Last but not the least, I would like to emphasize that this work was heavily influenced by my mathematical interactions with Boris Feigin, from the first encounter as a high school student in 1996 to the present day. His way of thinking about mathematics shaped my worldview in more ways than I would ever be able to describe; I am excited to dedicate this paper to him on the occasion of his 70th birthday, and to wish him many happy returns of the day.  

\printbibliography

\end{document}